\documentclass[journal]{IEEEtran}

\usepackage{cite}
\usepackage{bm}
\usepackage{amsmath,amssymb,amsfonts}

\usepackage{mathtools}
\usepackage{algorithm}
\usepackage{algcompatible}
\usepackage{graphicx}
\usepackage{textcomp}
\usepackage[hidelinks]{hyperref}
\usepackage{savefnmark}
\usepackage{cite,url,algorithm}
\usepackage{algorithmicx}
\usepackage{algpseudocode}
\usepackage{tikz}
\usepackage{pgfplots}
\usepackage{subfigure}

\DeclareMathOperator{\Ima}{Im}

\DeclarePairedDelimiter\floor{\lfloor}{\rfloor}

\definecolor{mycol1}{RGB}{0, 103, 165}
\definecolor{mycol2}{RGB}{38, 115, 77}
\usepackage{hyperref} 
\hypersetup{
	colorlinks = true,
	citecolor = mycol2,
	linkcolor = mycol1,
	urlcolor = black
}

\graphicspath{{figures/}}

\let\labelindent\relax
\usepackage{enumitem}

\newcommand{\beqar}{\begin{eqnarray}}
\newcommand{\eeqar}{\end{eqnarray}}
\newcommand{\beqarno}{\begin{eqnarray*}}
\newcommand{\eeqarno}{\end{eqnarray*}}
\newcommand{\ba}[1]{\begin{array}{#1}}
\newcommand{\ea}{\end{array}}
\newcommand{\st}{\mathop{\rm s.t.}\nolimits}
\newcommand{\rr}{{\mathbb R}}
\newcommand{\card}{\mathop{\rm card}\nolimits}

\newcommand{\II}{{\mathcal I}}
\definecolor{annotation_bg}{RGB}{204,204,255}
\definecolor{annotation_fg}{RGB}{128,0,0}

\newcommand{\matrice}[2]{\left[\hspace*{-.1cm}\ba{#1} #2 \ea\hspace*{-.1cm}\right]}
\newcommand{\smallmat}[1]{\left[ \begin{smallmatrix}#1 \end{smallmatrix} \right]}
\newcommand{\Aa}{\mathcal{A}}
\newcommand{\BB}{\mathcal{B}}
\newcommand{\CC}{\mathcal{C}}
\newcommand{\QQ}{\mathcal{Q}}
\newcommand{\RR}{\mathcal{R}}
\newcommand{\PP}{\mathcal{P}}
\newcommand{\KK}{\mathcal{K}}
\newcommand{\rank}{\mathop{\rm rank}\nolimits}
\newcommand{\one}{\mathop{{\rm 1}\mskip-5.0mu{\rm I}}\nolimits}
\newcommand{\diag}{\mathop{\rm diag}\nolimits}
\newcommand{\blkdiag}{\mathop{\rm blockdiag}\nolimits}
\newcommand{\eqdef}{\triangleq}
\newcommand{\qed}{\hfill{\small $\blacksquare$}}

\ifx\example\undefined
\newtheorem{theorem}{Theorem}[section]
\newtheorem{exampleenv}{Example}[section]
\newcommand{\example}[2]{
    \begin{exampleenv}
        {
            \label{#1}
            {#2}
        }
    \end{exampleenv}
}

\newtheorem{lemmaenv}{Lemma}[section]
\newcommand{\lemma}[2]{
	\begin{lemmaenv}
		{\em
			\label{#1}
			{#2}
		}
	\end{lemmaenv}
}

\newtheorem{propositionenv}{Proposition}[section]
\newcommand{\proposition}[2]{
	\begin{propositionenv}
		{\em
			\label{#1}
			{#2}
		}
	\end{propositionenv}
}

\newtheorem{remarkenv}{Remark}[section]
\newcommand{\remark}[2]{
    \begin{remarkenv}
        {
            \label{#1}
            {#2}
        }
    \end{remarkenv}
}
\fi

\begin{document}

\author{A. Bemporad\thanks{A. Bemporad is with the IMT School for Advanced Studies Lucca, Italy. Email: \texttt{alberto.bemporad@imtlucca.it}} and G. Cimini\thanks{G. Cimini is with ODYS S.r.l., Italy. Email: \texttt{gionata.cimini@odys.it}}}

\title{Reduction of the Number of Variables in Parametric Constrained Least-Squares Problems}

\maketitle

\begin{abstract}
For linearly constrained least-squares problems that depend on a vector of parameters, this paper proposes techniques for reducing the number of involved optimization variables. After first eliminating equality constraints in a numerically robust way by QR factorization, we propose a technique based on singular value decomposition (SVD) and unsupervised learning,
that we call $K$-SVD, and neural classifiers to automatically partition the set of parameter vectors in $K$ nonlinear regions in which the original problem is approximated by using a smaller set of variables. For the special case of parametric constrained least-squares problems that arise from model predictive control (MPC) formulations, we propose a novel and very efficient QR factorization method for equality constraint elimination. Together with SVD or $K$-SVD, the method provides a numerically robust alternative to standard condensing and move blocking, and to other complexity reduction methods for MPC based on basis functions. We show the good performance of the proposed techniques in numerical tests and in a linearized MPC problem of a nonlinear benchmark process.
\end{abstract}


\section{Introduction}
\label{sec:intro}
Several problems in engineering can be reformulated as parametric constrained least-squares  (pCLS) problems in which the matrices defining the cost function and the equality and inequality constraints depend on a set of parameters. Examples range from control engineering, in particular model predictive control (MPC)~\cite{MRD18}, filtering and smoothing \cite{RRM03}, financial engineering~\cite{SB09,BPG11,BBDKKNS17}, to mention a few.
A common feature in such applications is that the pCLS problem must be solved quickly and in a numerically robust way, often in simple embedded devices. This asks for methods that can reformulate, and possibly approximate, the problem by reducing the number of optimization variables, and that aim at good numerical conditioning, so to ease the numerical procedure employed to solve the reduced problem.

In particular, in MPC formulations the equality constraints are usually eliminated by substituting the predicted state as a function of the applied inputs, a.k.a. \emph{condensing}~\cite{FJ13}, that, as we will discuss later, can lead to very badly conditioned problems. In addition, the number of optimization variables is typically reduced by using \emph{move blocking}~\cite{BMR20,CGKM07,GI10,SM15}. This corresponds to keeping the input signal constant (``blocked'') between pre-specified prediction steps; typically the input signal is free to move during the first few prediction steps and then kept constant. The blocking scheme could be also modified in real-time, such as with the heuristic method proposed in \cite{SVDB16}. Blocking moves, however, complicate the recursive structure of the problem, therefore making efficient solution methods conceived for the non-condensed problem~\cite{Wri19} not directly applicable. Moreover, 
move blocking may prevent the application of involved blocking strategies such as
the aforementioned~\cite{SVDB16}. 

The so-called \emph{partial condensing}~\cite{A15} lies in between the condensed and non-condensed forms. It is a reformulation of the MPC optimization problem 
in which just a subset of the equality constraints 
are eliminated by replacing some of the states with the corresponding state responses.
Such a reformulated quadratic programming (QP) problem has a \emph{block sparse} structure, which reduces the number of optimization variables while still allowing the exploitation of efficient sparse linear algebra. For a benchmark comparisons between partial condensed and non-condensed forms, the reader is referred to~\cite{KFZD18}, where improvements are shown for open-source solvers commonly used in sparse problems. 

Alternatives to the standard condensing methods
have been investigated with the aim of preserving the sparsity of the lower dimensional problem, so to exploit tailored solvers. This is typically denoted as \emph{sparse condensing}. For instance, in~\cite{DLM17} the authors propose the use of Turnback algorithm~\cite{BHKLPW85} to find a banded null basis, and suggest to perform a open-loop simulation 
with a zero input excitation 
to get a particular solution. 
In \cite{JKC12}, a banded formulation is obtained by computing the deadbeat feedback gain, under the assumption of null controllability. The method proposed in \cite{YMDJH19} relaxes such an assumption and the basis for equality constraints is computed through deadbeat responses.
One drawback of these approaches is that in case of unstable models the generated open-loop solution 
may contain numbers that exponentially increase with the prediction horizon. 

Principal Component Analysis (PCA) was also proposed for input trajectory parameterization in MPC formulations. In~\cite{OW14}, the authors propose to construct a basis 
from the principal components of the Hessian of the condensed form. Alternatively, such a basis can be derived from the principal components of the open-loop response matrix, as in \cite{RGSF04}, which however limits its application to the only case where standard condensing 
is used to eliminate equality constraints. 
A different approach, proposed in \cite{UMJ12}, is to apply PCA to a matrix that collects the snapshots of previous control inputs already applied to the real plant, or to a simulation model. A drawback of all the above methods is that, if the matrices of the linear prediction model

are time-varying, PCA must be performed in real-time, making the execution of the overall MPC algorithm computationally intractable in many applications. 
Additionally, the method does not extend to pCLS problems
in which the equality constraints are eliminated by using more
general numerical methods, like the ones described in~\cite[Ch.~20--22]{LH74}. 

\subsection{Contribution}
In this paper we address the issue of reducing the number of optimization variables in pCLS problems by eliminating equality constraints, and of further eliminating degrees of freedom, possibly sacrificing optimality, to simplify the resulting optimization problem. First, we propose a method based on a computationally efficient QR factorization of the matrix associated with equality constraints, that in case of pCLS problems arising from MPC is a much more numerically stable 
method to eliminate equalities 
than standard condensing, in particular when the dynamics 
given by the prediction model are unstable. 

A second contribution is an offline procedure based on PCA, unsupervised learning, and multiclass classification to reduce the number of remaining optimization variables. We call the unsupervised learning procedure $K$-SVD, as it is an extension of SVD to determine $K$ different sets of principal directions, rather than just one as in standard linear PCA. As each sample vector is associated with a corresponding
parameter vector, $K$-SVD determines a nonlinear partition 
of the parameter space into regions, and provides a set of principal directions in each region.
When $K$-SVD is applied to pCLS problems, each sample vector is an optimal solution of the pCLS problem for a corresponding value of the parameter vector, and the method returns $K$ different approximate reformulations of the pCLS problem.

For pCLS problems arising from MPC, we also address issues of recursive feasibility of the proposed scheme, in order to make sure that reducing the number of free optimization variables
does not lead to infeasible reduced-order problems, which is particularly important in an MPC setting. 

Finally, we provide evidence of the benefits of the proposed methods in numerical experiments.

\subsection{Notation}
Given a finite set $I=\{i_1,\ldots,i_N\}$, $\card(I)$ denotes its number $N$ of elements (cardinality). Given a vector $v\in\rr^n$, $\|v\|_2$ is the Euclidean norm
of $v$, $|v|$ is the component-wise absolute value of $v$, $\begin{bmatrix}
v\end{bmatrix}_+$ is the projection of $v$ onto the non-negative orthant,
and $\diag(v)$ is the diagonal matrix of $\rr^{n\times n}$ formed from $v$. Given a matrix $A\in\rr^{n\times m}$, $\|A\|_F$ is the Frobenius norm of $A$, $\|A\|_F^2=\sum_{i=1}^n\sum_{j=1}^mA_{ij}^2$, $\ker(A)$ is the null space of $A$, $\Ima(A)$ the range space of $A$, and $\rank(A)$ the rank of $A$. The sub-matrix of $A$ obtained by collecting its entries whose row-index ranges from $i$ to $j$ and column-index from $h$ to $k$ is denoted by $A_{i:j,k:h}$. We define as $\kappa(A)=\|A\|\|A^+\|$ the generalized condition number of $A$ with respect to the spectral norm, with $A^+$ the Moore-Penrose inverse of $A$.
The matrix $0_{m}\in\rr^{n\times m}$ is the zero matrix with $m$ columns,
or the zero row vector if $n=1$.
The matrix $I_n\in\rr^{n\times n}$ is the identity matrix of dimension $n$.
 We denote by $\mathcal{U}(a,b)$, $a,b\in\rr$, $b>a$, the uniform distribution of real numbers in the interval $[a,b]$, and by $\mathcal{D}(c,d)$, $c,d\in\mathbb{N}_0$, $d>c$ the uniform discrete distribution in the interval $[c,d]$.

\section{Problem formulation}
\label{sec:problemFormulation}
Consider the following parameter-dependent constrained least-squares (CLS) problem
\begin{subequations}
    \beqar
    \min_{z}&&\frac{1}{2}\|A(\theta)z-b(\theta)\|_2^2\label{eq:pCLS-cost}\\
    \st&&  C(\theta)z=e(\theta)\label{eq:pCLS-eq}\\
    &&G(\theta)z\leq g(\theta)\label{eq:pCLS-ineq}
    \eeqar%
    \label{eq:pCLS}%
\end{subequations}
where $z\in\rr^\ell$ is the optimization vector and $\theta\in\rr^p$ is the vector of parameters
defining the problem instance,
$A(\theta)\in\rr^{n_c\times\ell}$, 
$b(\theta)\in\rr^{n_c}$, 
$C(\theta)\in\rr^{n_e\times\ell}$, $e(\theta)\in\rr^{n_e}$,
$G(\theta)\in\rr^{n_i\times\ell}$, and $g(\theta)\in\rr^{n_i}$. 
For simplicity, we assume that $\rank(C(\theta))=n_e$, $\forall \theta\in\rr^p$, although
this assumption can be easily relaxed
(see \hyperref[rmk:qr]{Remark~\ref{rmk:qr}} below).
Our goal is to find numerically
efficient and robust ways to solve~\eqref{eq:pCLS}, including methods
to approximate the problem so that it can be solved with respect to a \emph{reduced number of optimization
variables}. Since now on, for simplicity of notation
we will drop the dependence on $\theta$ where obvious.

\subsection{Variable reduction by elimination of equality constraints}
\label{sec:var_elim}
We start reducing the number of variables by eliminating the equality constraints~\eqref{eq:pCLS-eq}. Several methods exist to handle equality-constrained least squares problems~\cite[Ch.~20--22]{LH74}. 
We consider the following numerically-robust variable-elimination method based on the 
QR factorization
\begin{equation}
    C'=QR,\quad Q=\matrice{cc}{Q_1& Q_2},\quad R=\matrice{c}{R_1\\0}
    \label{eq:QR-fact-C}
\end{equation}
where $Q$ is an orthogonal matrix, $Q'Q=I$, $Q_1\in\rr^{\ell\times n_e}$, $Q_2\in\rr^{\ell\times (\ell-n_e)}$, and $R$ is upper triangular, with $R_1\in\rr^{n_e\times n_e}$. The columns of matrix $Q_2$ provide a basis of $\ker(C)$, as 
\[
    CQ_2=\matrice{cc}{R_1'&0}\matrice{c}{Q_1'\\Q_2'}Q_2=\matrice{cc}{R_1'&0}\matrice{c}{0\\I}=0
\]
Consider now the change of variables 
\[
    y=\matrice{c}{\bar s\\s},\quad \bar s=Q_1'z,\quad s=Q_2'z
\]
Since $z=Qy$, we get $Cz=R'Q'Qy=\matrice{cc}{R_1'&0}y=R_1'\bar s=e$.
As we assumed $\rank(C)=n_e$, matrix $R_1$ is nonsingular and hence
\begin{equation}
    \bar s=(R_1')^{-1}e,\ s~\mbox{free}
    \label{eq:s_bar}
\end{equation}
where $s\in\rr^{\ell-n_e}$ is the vector of remaining optimization variables. By substituting 
\begin{subequations}
    \beqar
    z&=&Q_2s+\bar z\label{eq:z_of_s-a}\\
    \bar z&\eqdef& Q_1\bar s\label{eq:z_of_b}
    \eeqar
    \label{eq:z_of_s}%
\end{subequations}
the following CLS problem \emph{without} equality constraints
\begin{subequations}
    \beqar
    \min_s&&\frac{1}{2}\|AQ_2s-(b-A\bar z)\|_2^2\label{eq:pCLS-cost-r}\\
    \st&&  GQ_2s\leq g-G\bar z\label{eq:pCLS-ineq-r}
    \eeqar%
    \label{eq:pCLS-r}%
\end{subequations}
is equivalent to~\eqref{eq:pCLS}, where in~\eqref{eq:pCLS-r} all constant matrices/vectors are in general a function of $\theta$. Clearly, $\bar{z}$ satisfies~\eqref{eq:pCLS-eq}.

The QR elimination method described above enjoys the following property:
\proposition{prop:qrprop}{
    Vector
    $\bar z$ solves the minimum norm problem 
    \begin{equation}
        \bar z=\arg\min_z \|z\|^2\quad\,\, \textrm{subject to}\,\, Cz=e
        \label{eq:min-norm}
    \end{equation}
}
\begin{proof}
    The optimality conditions from Problem~\eqref{eq:min-norm} are $2z+C' z_D=0$, $Cz=e$,
    which gives $z_D=-2(CC')^{-1}e$, and hence
    $z=C'(CC')^{-1}e$ $=$ $QR(R'R)^{-1}e$ $=$ $Q_1R_1R_1^{-1}(R_1')^{-1}e=\bar z$.
\end{proof}

The orthogonal unit vectors given by the columns of $\begin{bmatrix}Q_1 & Q_2\end{bmatrix}$ and the associated coordinate transformation $z=\bar z+Q_2s$ are ideal for numerical robustness. Consider a generic coordinate transformation $z=Ys_y+Zs$, where $Y$ and $Z$ denote any matrices whose columns form a basis respectively for $\Ima(C)$ and $\ker(C)$. Being $CY$ non-singular, any vector $z=Y(CY)^{-1}e+Zs$ satisfies~\eqref{eq:pCLS-eq}, and one has to select $Y$ in such a way that $CY$ is well conditioned. \hyperref[prop:qrprop]{Proposition~\ref{prop:qrprop}} shows that if $Y(CY)^{-1}e$ solves \eqref{eq:min-norm} then $Y(CY)^{-1}=C'(CC')^{-1}$ which makes the condition number of $CY$ independent from the basis selection. The pair $(Y,Z)$ such that \hyperref[prop:qrprop]{Proposition~\ref*{prop:qrprop}} holds true is not unique. Among the possible choices, letting $(Y,Z)\triangleq(Q_1,Q_2)$ guarantees also the best bound on $i)$ $\kappa(CY)$ as $\kappa(CQ_1)=\kappa(R_1)=\kappa(C)$, $ii)$ $\kappa(AZ)$ as $\kappa(AQ_2)\le\kappa(A)$ if $n_c\ge \ell$ and $A$ non-singular, $iii)$ $\kappa(GZ)$ as $\kappa(GQ_2)\le\kappa(G)$ if $n_i\ge\ell$ and $G$ non-singular. We note that $AZ$ and $GZ$ are the matrices forming the reduced CLS problem~\eqref{eq:pCLS-r}, therefore the coordinate transformation $z=\bar z+Q_2s$ improves the numerical robustness of both variables elimination and CLS problem optimization.

\remark{rmk:qr}{The variable elimination method for equality constraints can be easily generalized to the case in which $\rank(C)=n_3< n_e$. Let $C'P=Q\begin{bmatrix}
        R_{11} & R_{12}\\ 0 & R_{22}
    \end{bmatrix}$ be a rank-revealing QR factorization of $C'$, with $R_{11}\in\rr^{n_3\times n_3}$, $R_{22}\in\rr^{(n_e-n_3)\times(n_e-n_3)}$ and   $P\in\rr^{n_e\times n_e}$ a permutation matrix such that $||R_{2,2}||\le \epsilon$, with $\epsilon$ an arbitrary small tolerance. Let $P=\begin{bmatrix}
        P_1 & P_2
    \end{bmatrix}$ with $P_1\in\rr^{n_e\times n_3}$ and $P_2\in\rr^{n_e\times (n_3-n_e)}$. Then, equality constraints can be eliminated by replacing the fixed variables in~\eqref{eq:s_bar} with $\bar s=(R_{11}')^{-1}P_1'e$. Similarly, the method can be also generalized to the less common case in which $\ell\le n_e$ and, possibly, $\rank(C)<n_e$.\qed}

\section{Variable reduction by SVD}
\label{sec:basis-SVD}
After eliminating equality constraints, we want to attempt to further reduce the complexity 
of solving~\eqref{eq:pCLS-r} by lowering the number of optimization
variables from $n$ to $m$, possibly at the price of introducing suboptimality and infeasibility
(we will handle feasibility issues in \hyperref[sec:feasibility]{Section~\ref{sec:feasibility}}). To this end, let us consider a basis $\Phi=[\phi_1\ \ldots\ \phi_m]\in\rr^{n\times m}$ and a constant vector $\phi_0\in\rr^{n}$ that we use to parameterize the vector $s$ of free variables as an affine function of a new vector $v\in\rr^m$ of free optimization variables
\begin{equation}
    s=\phi_0+\sum_{i=1}^m\phi_iv_i=\phi_0+\Phi v
    \label{eq:s-basis}
\end{equation}
We assume that matrix $\Phi$ is full column rank, otherwise
some degrees of freedom $v_i$ would be wasted. 

In the presence of inequality constraints~\eqref{eq:pCLS-ineq}, after substituting~\eqref{eq:s-basis} and optimizing with respect to $v$ we would like to recover at best the solution $s^*$ of the original inequality-constrained least-squares problem~\eqref{eq:pCLS-r}. However, as we also want to recover the exact solution of~\eqref{eq:pCLS} when
all inequality constraints~\eqref{eq:pCLS-ineq} are inactive and $m<n$, 
for any given new value of $\theta$,
before applying~\eqref{eq:s-basis},
we first compute the unconstrained solution of~\eqref{eq:pCLS-cost-r}
\begin{equation}
    s^*_u=(Q_2'A'AQ_2)^{-1}Q_2'(b-A\bar z)
    \label{eq:unconstrained_sol}
\end{equation}
and check if constraints~\eqref{eq:pCLS-ineq-r} are satisfied for $s=s^*_u$. In case they are,
we can immediately retrieve the optimal solution $z^*_u=\bar z+Q_2s^*_u$.
Otherwise, we substitute~\eqref{eq:s-basis} in~\eqref{eq:pCLS-r} and solve the reduced CLS problem
\begin{subequations}
    \beqar
    v^*=\arg\min_{v\in\rr^m} &\frac{1}{2}\|AQ_2\Phi v-(b-Az_0)\|_2^2\label{eq:pCLS-cost-r-v}\\
    \st&  GQ_2\Phi v\leq g-Gz_0\label{eq:pCLS-ineq-r-v}
    \eeqar%
    \label{eq:v-star}%
\end{subequations}
where $z_0=\bar z+Q_2\phi_0$, and then set
\begin{equation}
    z^*=Q_2\Phi v^*+z_0
    \label{eq:zstar-r}
\end{equation}

\subsection{Choice of basis}
Let us focus on finding a basis $\Phi$ that enables us to reconstruct $s^*$ under inequality
constraints~\eqref{eq:pCLS-ineq-r} at best when the unconstrained solution
$s^*_u$ in~\eqref{eq:unconstrained_sol} is infeasible. Due to the dependence of~\eqref{eq:pCLS} on the vector $\theta$ of parameters, $s^*$ is also a function of $\theta$.

A standard approach for dimensionality reduction is to resort to
principal component analysis (PCA) to identify an optimal basis $\Phi$ for a given set of values $\theta_i$ of the parameter vector, $i=1,\ldots,M$, as follows. For each $\theta_i$, let $s_i^*$ be the constrained solution of~\eqref{eq:pCLS-r}
when $\theta=\theta_i$. Let $\bar s_m=\frac{1}{M}\sum_{i=1}^Ms_i^*$ be the mean of the collected vectors $s_i^*$, and let 
$S=[s_1^*-\bar s_m\ \ldots\ s_M^*-\bar s_m]'$, $S\in\rr^{M\times n}$. Consider the economy-size singular value decomposition (SVD) of $S$
\begin{equation}
    S=U\Sigma V'
    \label{eq:svd-W}
\end{equation}
where $\Sigma=\diag([\sigma_1\ \ldots\ \sigma_{n}]')$, $\sigma_i\geq\sigma_j$ for $i\leq j$.
Matrix $V=[V_1\ \ldots\ V_n]\in\rr^{n\times n}$ is orthogonal, $V'V=I$, and provides the principal directions. Matrix $U\in\rr^{M\times n}$ collects $n$ columns of an orthogonal matrix,
such that $U_i\Sigma$ are the principal components of $s_i^*-\bar s_m$ along the principal directions
$V_1$, $\ldots$, $V_m$, and $U_i$ is the $i$-th row of $U$.
Given the number $m$ of degrees of freedom we want to allow in the constrained least-squares problem~\eqref{eq:pCLS-ineq-r-v} by restricting $s$ as in~\eqref{eq:s-basis}, we set
\begin{subequations}
    \beqar
    \Phi&=&[V_1\ \ldots\ V_m]\\
    \phi_0&=&\bar s_m\label{eq:PCA-basis-phi0}
    \eeqar%
    \label{eq:PCA-basis}%
\end{subequations}
Note that, as suggested above, problem~\eqref{eq:v-star} is not solved when the solution $s_u^*$ is already feasible with respect to the constraints~\eqref{eq:pCLS-ineq-r-v}, and therefore
the basis vectors in~\eqref{eq:PCA-basis} will not be used in such a case. Hence,
to form matrix $S$ we only need to collect samples $\theta_i$ such that the corresponding unconstrained optimizer $s_u^*$ violates~\eqref{eq:pCLS-ineq-r-v}.
To this end, we select vectors $\theta_i$ such that the corresponding optimal Lagrange multipliers $\lambda_i$ associated with the inequality constraints~\eqref{eq:pCLS-ineq-r} are above a certain
threshold $\epsilon_\lambda>0$.

\subsection{Feasibility}
\label{sec:feasibility}
Because of the reduction of the number of degrees of freedom, even if problem~\eqref{eq:pCLS}
is feasible it might happen that~\eqref{eq:v-star} is infeasible.
In this section we discuss a few ways to address this issue.

First, assuming that we have a feasible vector $z_f$ available with respect to
the constraints~\eqref{eq:pCLS-eq}--\eqref{eq:pCLS-ineq}, or equivalently
a solution $s_f$ satisfying~\eqref{eq:pCLS-ineq-r}, we want to be able
to obtain $z_f$ from the parameterization~\eqref{eq:s-basis} in spite of the reduction
of the number of degrees of freedom. To achieve this, one can set $\phi_0=s_f$ instead of $\phi_0$ as in~\eqref{eq:PCA-basis-phi0}
in order to guarantee that $v=0$
is a feasible solution. A more flexible approach is to add 
$\bar s_m-s_f$ in the basis and parameterize $s$ as
\begin{equation}
    s=s_f+v_0(\bar s_m-s_f)+\sum_{i=1}^{m}v_i\phi_i
    \label{eq:s-basis-feas}
\end{equation}
where $v_0\in\rr$ is an extra degree of freedom. Again, the resulting
reduced CLS problem is feasible for $v=0$, where now $v\in\rr^{m+1}$,
and maintains the optimal principal component decomposition~\eqref{eq:PCA-basis}
for $v_0=1$.

A second approach is to replace~\eqref{eq:pCLS-ineq} with the following soft constraints
\begin{equation}
    G(\theta)z\leq g(\theta)+V_g(\theta)\zeta
    \label{eq:pCLS-ineq-soft}
\end{equation}
where $V_g(\theta)\in\rr^{n_i\times n_\zeta}$ is a matrix whose entries are all zero except one per row which is positive. The vector of \emph{slack variables} $\zeta\in\rr^{n_{\zeta}}$ is penalized by changing~\eqref{eq:pCLS-cost} to
\begin{equation}
    \min_{z}\frac{1}{2}\left\|\matrice{cc}{A(\theta) & 0\\0 & \Lambda_\zeta(\theta)}\matrice{c}{z\\\zeta}-\matrice{c}{b(\theta)\\0}\right\|_2^2
    \label{eq:pCLS-cost-soft}
\end{equation}
where $\Lambda_\zeta(\theta)\in\rr^{n_\zeta\times n_\zeta}$ is a diagonal matrix collecting the weights used to penalize $\zeta$. In this case, constraints~\eqref{eq:pCLS-ineq-soft} become
\begin{equation}
    G(\theta)Q_2(\theta)\Phi v\leq g(\theta)-G(\theta)z_0+V_g(\theta)\zeta
    \label{eq:pCLS-ineq-r-v-soft}
\end{equation}
When the original CLS problem is formulated with soft-constraints~\eqref{eq:pCLS-ineq-soft}, 
an alternative approach is to adopt the parameterization introduced in~\eqref{eq:PCA-basis} and keep both $v$ and $\zeta$ as optimization variables. In case $\zeta$ is a vector ($n_\zeta>1$), the number of slacks can be also arbitrarily shrunk to $\bar n_\zeta$, $1\leq \bar n_\zeta<n_\zeta$, by right multiplying matrix $V_g(\theta)$ by a binary matrix $E_\zeta$, $E_\zeta\in\{0,1\}^{n_\zeta\times\bar n_\zeta}$, with
all columns of $E_\zeta$ being nonzero. For instance, for the minimum value $\bar n_\zeta=1$, we set $E_\zeta=[1\ \ldots\ 1]'$.

\example{ex:PCA-basis}{
    Consider a parameter-dependent CLS problem with $n=20$, $n_\theta=4$, $n_c=20$, 
    $A(\theta)\equiv A$, $b(\theta)=b+F\theta$, with the entries of $A$, $b$, and $F$ 
    randomly generated from a normal distribution, and $G(\theta)=[I\ -I]'$, $g(\theta)=[1\ \ldots\ 1]'$ (box constraints).
    After generating $M=1000$ random samples of $\theta$ and collecting the corresponding
    optimal solutions we apply the SVD decomposition~\eqref{eq:PCA-basis} for values of $m$ ranging between $1$ and $n-1$. Then, we generate further $N_{\rm val}=100$ vectors $\theta$ for validation. In both cases, in collecting the parameter vectors $\theta_i$ we discard those whose associated optimal dual variables are all smaller than $\epsilon_\lambda=10^{-3}$, to ensure that at least one inequality constraint is active at optimality. For each $\theta$, we solve the reduced CLS problem~\eqref{eq:v-star} with respect to $v\in\rr^m$ and to the additional slack variable $\zeta\in\rr$ as in~\eqref{eq:pCLS-cost-soft}
    with $\Gamma_\zeta=10^5$ and in~\eqref{eq:pCLS-ineq-soft} $V_g=[1\ \ldots\ 1]'$.
    
    \hyperref[fig:PCA-solution-error]{Figure~\ref{fig:PCA-solution-error}} shows the results obtained in terms
    of the relative 
    optimality error $\frac{\|Az^*-b-F\theta\|-\|Az^*_r-b-F\theta\|}{\|Az^*-b-F\theta\|}$,
    the relative error $\frac{\|z^*-z^*_r\|}{\|z^*\|}$, and
    the maximum relative constraint violation $\max_i\left\{\frac{G_iz^*_r-g_i}{|g_i|}\right\}$,
    where all norms are Euclidean norms. As expected, the higher the number $m$ of basis vectors 
    the better is the quality of the approximation $z_r$. 
\qed}

\begin{figure}[t]
    \centerline{\includegraphics[width=\hsize]{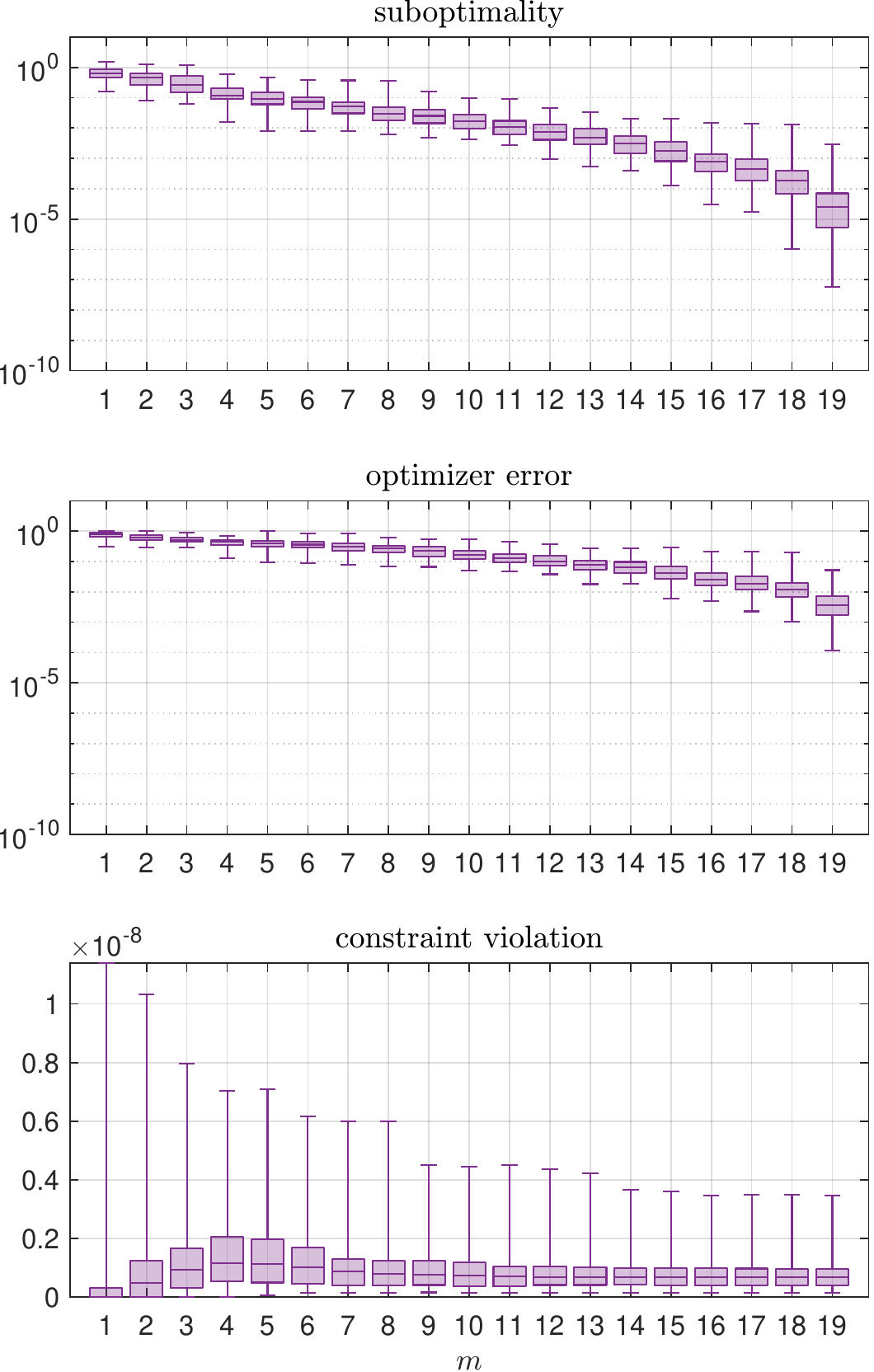}}
    \caption{Distribution of the errors induced by using~\eqref{eq:PCA-basis}, as a function of $m$:
        relative optimality error (top plot), relative difference of optimizers (mid plot), relative
        maximum violation of the inequality constraints (bottom plot) 
        over 100 validation tests.
    }    
    \label{fig:PCA-solution-error}
\end{figure}

\section{Generalization to parameter-dependent bases}
\label{sec:K-SVD}
The method described in \hyperref[sec:basis-SVD]{Section~\ref{sec:basis-SVD}} determines a basis $\Phi$
to use for all parameter vectors $\theta\in\rr^p$. In order to gain more flexibility, we propose to make $\Phi$ a piecewise-constant function of $\theta$ by using unsupervised learning and multiclass classification methods, as we describe next.

\begin{algorithm}[t]
    \caption{$K$-means in optimizer space.}
    \label{algo:K-means}
    ~~\textbf{Input}: Optimizer samples $s_1^*,\ldots,s_M^*\in\rr^{n}$, 
    number $K$ of clusters, number $m$ of basis elements.
    \vspace*{.1cm}\hrule\vspace*{.1cm}
    \begin{enumerate}[label*=\arabic*., ref=\theenumi{}]
        \item Create random partition $I_1,\ldots,I_K$, $\cup_{i=1}^KI_i=\{1,\ldots,M\}$, $I_i\cap I_j=\emptyset$, $\forall i,j=1,\ldots,K$, $i\neq j$;\label{step:algo1-1}
        \item \textbf{repeat}:
        \begin{enumerate}[label=\theenumi{}.\arabic*., ref=\theenumi{}.\arabic*]
            \item \textbf{for} $j=1,\ldots,K$ \textbf{do}:
            \begin{enumerate}[label=\theenumii{}.\arabic*., ref=\theenumii{}.\arabic*]
                \item $M_j\leftarrow$ number of elements of $I_j$;
                \item $\phi_0^j\leftarrow\frac{1}{M_j}\sum_{i\in I_j}s^*_i$;
            \end{enumerate}
            \item \textbf{for} $i=1,\ldots,M$ \textbf{do}:
            \begin{enumerate}[label=\theenumii{}.\arabic*., ref=\theenumii{}.\arabic*]
                \item Reassign $s_i^*$ to cluster $j$ such that
                \begin{equation}
                    j=\arg\min_{j\in\{1,\ldots,K\}}\|s_i^*-\phi_0^j\|_2^2
                    \label{algo:K-means-reassignment}
                \end{equation}
            \end{enumerate}            
        \end{enumerate}
        \item \textbf{until} convergence;
        \item \textbf{for} $j=1,\ldots,K$ \textbf{do}:
        \begin{enumerate}[label=\theenumi{}.\arabic*., ref=\theenumi{}.\arabic*]
            \item Compute basis $\Phi_j\in\rr^{n\times m}$ using SVD
            of matrix $S_j$ obtained by collecting $s_i^*-\phi_0^j$ for $i\in I_j$;
            \label{algo:K-means:SVD}
        \end{enumerate}
        \item \textbf{end}.
    \end{enumerate}
    \vspace*{.1cm}\hrule\vspace*{.1cm}
    ~~\textbf{Output}: Clusters $I_1,\ldots,I_K$, corresponding basis
    $\Phi_1,\ldots,\Phi_K$ and offset vectors
    $\phi_0^1,\ldots,\phi_0^K$.
\end{algorithm}

Assume $M$ samples $s_i^*$ have been collected for the corresponding set $\Theta=\{\theta_i\}_{i=1}^M$ of parameter vectors, as
described in \hyperref[sec:basis-SVD]{Section~\ref{sec:basis-SVD}}, and that we only allow $K$ different matrices  $\Phi_1,\ldots,\Phi_K\in\rr^{n\times m}$ and vectors
$\phi_0^1,\ldots,\phi_0^K\in\rr^n$ for the parameterization of $s$. A first method is to simply perform $K$-means \cite[Algorithm 14.1]{HTF09} on $\Theta$ to get $K$ clusters and repeat the approach
of \hyperref[sec:basis-SVD]{Section~\ref{sec:basis-SVD}} on each cluster. Then, the clusters can be separated
for example by using piecewise linear separation (the Voronoi diagram of the centroid
of the clusters, robust linear programming~\cite{BM94}, or one of the efficient algorithms proposed in~\cite{BPB16a}) in order to define a function that, for each given
$\theta\in\rr^p$, returns the corresponding basis $\Phi_j$ and $\phi_0^j$ to use. The main drawback of this approach is that clustering would be done only based on the Euclidean distance between two parameter
vectors, independently on the values of the corresponding optimizer $s^*$.

Alternatively, one can perform $K$-means
on the set $\{s_i^*\}_{i=1}^M$ to get $K$ clusters $I_1,\ldots,I_k$ of indices,
$\cup_{i=1}^KI_i=\{1,\ldots,M\}$, $I_i\cap I_j=\emptyset$, $\forall i,j=1,\ldots,K$, $i\neq j$,
as described in Algorithm~\ref*{algo:K-means}. This is a classical $K$-means algorithm with random initialization of the clusters $I_j$, $j=1,\ldots,K$, whose convergence is tested
when the set $\{I_j\}$ of clusters is not changing between two consecutive iterations.
After executing \hyperref[algo:K-means]{Algorithm~\ref{algo:K-means}}, we compute a matrix $\Phi_j$ for each cluster by SVD. The cluster indices $\{I_j\}$ determined by Algorithm~\ref{algo:K-means}
also induce a partition of $\Theta$ in corresponding $K$ sets. Finding a (nonlinear) separation function in the $\theta$-space $\rr^p$ is a \emph{multiclass classification} problem~\cite{Aly05} (or \emph{binary classification} if $K=2$), that is the problem of determining a function $\psi:\rr^p\to\{1,\ldots,K\}$ that associates to any vector $\theta$ its corresponding class $j=\psi(\theta)$. 

The main drawback of running $K$-means
in the space of optimization vectors $s$ is that a small \emph{distance} between two optimizers $s_i^*$, $s_j^*$ does not necessarily implies that they can be well approximated by the same basis. For example, $s_i^*=[10\ 0 \ldots\ 0]'$ is much ``closer'' to $s_j^*=[100\ 0 \ldots\ 0]'$
than to $s_h^*=[10\ 1 \ldots\ 0]'$, as $s_i^*,s_j^*$ are both multiple of the same vector.
In light of the above considerations, we propose the variant of $K$-means described in Algorithm~\ref*{algo:K-basis}, that we call \hyperref[algo:K-basis]{$K$-SVD} algorithm, to perform clustering of the index set $\{1,\ldots,M\}$ by reassigning each vector $s_i^*$ to the cluster $I_j$ whose current basis $\Phi_j$ and offset $\phi_0$ best represent it in a least-squares sense. 

\begin{algorithm}[t]
    \caption{\textcolor{mycol1}{$K$-SVD} for parameter-dependent PCA.}
    \label{algo:K-basis}
    ~~\textbf{Input}:  Optimizer samples $s_1^*,\ldots,s_M^*\in\rr^{n}$, 
    number $m$ of basis elements, initial clusters $I_1,\ldots,I_K$, $\cup_{i=1}^KI_i=\{1,\ldots,M\}$, $I_i\cap I_j=\emptyset$, $\forall i,j=1,\ldots,K$, $i\neq j$.
    \vspace*{.1cm}\hrule\vspace*{.1cm}
    \begin{enumerate}[label*=\arabic*., ref=\theenumi{}]
        \item \textbf{repeat}:
        \begin{enumerate}[label=\theenumi{}.\arabic*., ref=\theenumi{}.\arabic*]
            \item \textbf{for} $j=1,\ldots,K$ \textbf{do}:
            \begin{enumerate}[label=\theenumii{}.\arabic*., ref=\theenumii{}.\arabic*]
                \item $M_j\leftarrow$ number of elements of $I_j$;
                \item $\phi_0^j\leftarrow\frac{1}{M_j}\sum_{i\in I_j}s^*_i$;
                \label{algo:K-basis:phi0}
                \item Compute basis $\Phi_j\in\rr^{n\times m}$ using the economy-size SVD 
                \begin{equation}
                    S_j=[U_1^j\ U_2^j]\matrice{cc}{\Sigma^j_1&0\\0&\Sigma^j_2}[\Phi_j\ V_2^j]
                    \label{eq:K-basis-SVD}
                \end{equation}
                where the rows of $S_j$ are $(s_i^*-\phi_0^j)'$, $i\in I_j$;
                \label{algo:K-basis:SVD}
            \end{enumerate}
            \item \textbf{for} $i=1,\ldots,M$ \textbf{do}:
            \begin{enumerate}[label=\theenumii{}.\arabic*., ref=\theenumii{}.\arabic*]
                \item \label{algo:K-basis-reassignment}
                Reassign $s_i^*$ to cluster $j$ such that
                \begin{equation}
                    j=\arg\min_{j\in\{1,\ldots,K\}}\left\{\min_v\|s_i^*-\Phi_jv-\phi_0^j\|_2^2\right\}
                    \label{eq:K-basis-reassignment}
                \end{equation}
            \end{enumerate}            
        \end{enumerate}
        \item  \label{algo:K-basis-stopping} \textbf{until} convergence;
        \item \textbf{end}.
    \end{enumerate}
    \vspace*{.1cm}\hrule\vspace*{.1cm}
    ~~\textbf{Output}: Clusters $I_1,\ldots,I_K$, corresponding bases
    $\Phi_1,\ldots,\Phi_K$ of orthogonal vectors and offset vectors
    $\phi_0^1,\ldots,\phi_0^K$.
\end{algorithm}

\hyperref[algo:K-basis]{$K$-SVD} requires an initial partition of $\{1,\ldots,M\}$ in
$K$ clusters. A possible approach is to create random clusters
$I_1,\ldots,I_K$ as in \hyperref[step:algo1-1]{Step~\ref{step:algo1-1}} of \hyperref[algo:K-means]{Algorithm~\ref{algo:K-means}},
or to get $I_1,\ldots,I_K$ by fully executing the $K$-means Algorithm~\ref*{algo:K-means} 
on $s_1^*,\ldots,s_M^*$.

The following result proves that Algorithm~\ref*{algo:K-basis} is indeed an algorithm,
in that it terminates in a finite number of steps to a local minimum of the problem
of finding the $K$ ``best'' bases.
\begin{theorem}
    \label{th:convergence}
    Let $M>n>m$. \hyperref[algo:K-basis]{$K$-SVD} converges in a finite number of steps
    to a local minimum of the following optimization problem
    \begin{subequations}
        \beqar
        \min_{J,\{\Phi_j,\phi_0^j\}_{j=1}^K} &\displaystyle{ \sum_{i=1}^M\min_v\|s_i^*-\Phi_{J(i)}v-\phi_0^{J(i)}\|_2^2}
        \label{eq:k-svd-cost}\\
        \st& \displaystyle{\phi_0^j=\frac{1}{M_j}\sum_{i\in I_j} s_i^*,\ j=1,\ldots,K}
        \label{eq:k-svd-constr}%
        \eeqar
        \label{eq:k-svd-problem}%
    \end{subequations}
    where $J\in\{1,\ldots,K\}^M$ is a sequence of cluster labels, namely
    $J(i)=j$ implies that $s_i^*$ belongs to cluster \#$j$, 
    $I_j\eqdef\{i\in\{1,\ldots,M\}:\ J(i)=j\}$, $M_i=\card(I_j)$,
    and where the columns of each matrix $\Phi_j$ are orthogonal.
\end{theorem}

\proof We want to show that Algorithm~\ref{algo:K-basis} is a coordinate-descent
method that solves problem~\eqref{eq:k-svd-cost} by iterating between
minimizing w.r.t. $\{\phi_0^j\}_{j=1}^K$, then $\{\Phi^j\}_{j=1}^K$, and 
then $J$.

Consider first the label vector $J$ and $\{\Phi^j\}_{j=1}^K$ fixed.
\hyperref[algo:K-basis:phi0]{Step~\ref{algo:K-basis:phi0}} determines the vectors $\{\phi_0^j\}_{j=1}^K$ 
such that the equality constraints~\eqref{eq:k-svd-constr} are satisfied.
Now consider both the label vector $J$ and $\{\phi_0^j\}_{j=1}^K$ 
fixed. We need to solve~\eqref{eq:k-svd-cost}
with respect to the bases $\{\Phi^j\}_{j=1}^K$, which is equivalent to
solving the following problem
\begin{equation}
    \min_{\{\Phi_j\}_{j=1}^K} \sum_{j=1}^K\left(\sum_{i\in I_j}\min_v\|s_i^*-\Phi_{j}v-\phi_0^{j}\|_2^2\right)
    \label{eq:k-svd-cost-sep}
\end{equation}
Clearly problem~\eqref{eq:k-svd-cost-sep} is separable in the $K$
independent minimization problems
\begin{equation}
    \min_{\Phi,\{v_i\}_{i\in I_j}}\sum_{i\in I_j}\|\bar s_i-\Phi v_i\|_2^2
    \label{eq:k-svd-cost-j}
\end{equation}
where $v_i\in\rr^m$ and we have introduced the simplified notation $\bar s_i\eqdef s_i^*-\phi_0^{J(i)}$. 
Let $\bar S=[\bar s_{i_1}\ \ldots\ \bar s_{i_{M_j}}]'$,
$\bar S\in\rr^{M_j\times n}$, be the matrix collecting the $M_j$ samples
$\bar s_i'$ of cluster $j$ as its rows, where $\{i_1,\ldots,i_{M_j}\}=I_j$,
and let $D=[v_{i_1}\ \ldots\ v_{i_{M_j}}]'$ be the corresponding matrix
of optimal coordinates, $D\in\rr^{M_j\times m}$. 
By Eckart-Young-Mirsky theorem~\cite{GHS87}, the following matrix optimization problem
\begin{equation}
    \hat S=\arg\min_{\rank \hat S\leq m}\|\bar S-\hat S\|_F^2
    \label{eq:k-svd-cost-frob}
\end{equation}
is solved
by $\hat S=U_1\Sigma_1 V_1'$, where $U\Sigma V'$ is the SVD decomposition
of $\bar S$, $U=[U_1\ U_2\ U_3]$, $U_1\in\rr^{M_j\times m}$, 
$U_2\in\rr^{M_j\times n-m}$, 
$U_3\in\rr^{M_j\times (M_j-n)}$, $U'U=I$, 
$\Sigma=\smallmat{\Sigma_1& 0\\0 & \Sigma_2\\0&0}$ is the matrix
of singular values, $\Sigma\in\rr^{M_j\times n}$, and $\Sigma_1\in\rr^{m\times m}$, 
$\Sigma_2\in\rr^{(n-m)\times (n-m)}$, while
$V=[V_1\ V_2]$, $V_1\in\rr^{n\times m}$ $V_2\in\rr^{n\times(n-m)}$, $V'V=I$.
By setting $\Phi_j=V_1$ and $D=U_1\Sigma_1\in\rr^{M_j\times m}$, we get
\[
\|\bar S-\hat S\|_F^2=\sum_{i\in I_j}\|\bar s_i-\Phi_j v_i\|_2^2
\]
and therefore $\Phi_j$ is a matrix of $\rr^{n\times m}$ with orthogonal columns 
that solves problem~\eqref{eq:k-svd-cost-j}.

For fixed $\{\Phi^j\}_{j=1}^K$ and $\{\phi_0^j\}_{j=1}^K$, clearly~\eqref{eq:K-basis-reassignment} determines the vector $J$
of labels that minimizes~\eqref{eq:k-svd-cost}, as for each sample $s_i^*$
the basis $\Phi_j$ and bias $\phi_0^j$ are chosen that give the least value
of $\min_v\|s_i^*-\Phi_jv-\phi_0^j\|_2^2$.

Having shown that \hyperref[algo:K-basis]{$K$-SVD} is a coordinate-descent
algorithm, the cost function~\eqref{eq:k-svd-cost} is monotonically non-increasing
at each iteration and lower-bounded by zero, so it converges asymptotically.
Moreover, as the number of possible
combinations $J$ are finite, Algorithm~\ref*{algo:K-basis} always terminates 
after a finite number of steps, under the assumption that in case of multiple optima
at \hyperref[algo:K-basis-reassignment]{Step~\ref{algo:K-basis-reassignment}} 
the optimizer $J$ is always chosen in accordance with a predefined criterion,
and that the SVD algorithm used at \hyperref[algo:K-basis:SVD]{Step~\ref{algo:K-basis:SVD}} returns
the same matrices $U,\Sigma,V$ for the same given input matrix $\bar S$.
\qed

Due to the finite termination property shown in \hyperref[th:convergence]{Theorem~\ref{th:convergence}},
a termination criterion at \hyperref[algo:K-basis-stopping]{Step~\ref{algo:K-basis-stopping}} is to stop
when the label vector $J$ does not change after one iteration. 
Note that Algorithm~\ref*{algo:K-basis} is only guaranteed to convergence to a local
minimum, whether this is also a global one depends on the initial 
guess $J$.
Moreover, the decrease of the cost function in~\eqref{eq:k-svd-cost} can be easily monitored. In fact, from the proof of Eckart-Young-Mirsky theorem for the Frobenius norm, we have that for a given label sequence $I$ the corresponding optimal cost~\eqref{eq:k-svd-cost-sep} is
\[
\sum_{j=1}^K\left(\sum_{i\in I_j}\min_v\|s_i^*-\Phi_{j}v-\phi_0^{j}\|_2^2\right)
=\sum_{j=1}^K\sum_{h=m+1}^{n}(\Sigma_{2ii}^j)^2,
\]
that is a value that can be immediately computed as a byproduct of the SVDs in~\eqref{eq:K-basis-SVD}. 

After running \hyperref[algo:K-basis]{$K$-SVD}, the samples $\theta_i$ get also labelled by the index sets $I_1,\ldots,I_K$. Several multiclass classification methods can be then used
to determine the classification function $\psi$, such as 
multicategory proximal support vector machines~\cite{FM05} or neural networks~\cite{OU07}.
In this paper we use $K$ one-to-all neural classifiers $\psi_i:\rr^{p}\to[0,1]$, $i=1,\ldots,K$,  and then define
the nonlinear separation function $\psi:\rr^{p}\to[0,1]$ such that $\psi(\theta)=\arg\max_{j=1\ldots K}\{\psi_j(\theta)\}$.

\remark{rmk:explicit}{
    A justification for associating a basis $\Phi$, $\phi_0$ to a region of the space of parameters $\theta$ stems from the multiparametric analysis of problem~\eqref{eq:pCLS-r}. For instance, in case matrices $A$ and $G$ do not depend on $\theta$ and $b(\theta)$, $g(\theta)$ are affine,~\eqref{eq:pCLS-r} is a multiparametric quadratic programming (mpQP) problem~\cite{BMDP02a}, whose optimal solution $s^*(\theta)$ is piecewise affine. Therefore, $s(\theta)=H_j\theta+h_j$ on each polyhedral region $P_j$ of the set of parameters $\theta$ get
    partitioned by the mpQP algorithm, $j=1,\ldots,n_p$. If the parametric solution $s(\theta)$ were known, $n_p\leq K$, $m\geq p$, for all $\theta\in P_j$ the parameterization~\eqref{eq:s-basis} would reproduce the optimal solution if one 
    computes the QR factorization of $H_j=Q_jR_j$ and sets $\Phi_j=[Q_j\ 0_{p-m}]$, $\phi_0^j=h_j$, and $v=[(R_j\theta)'\ 0_{p-m}]'$.
\qed}

\remark{rmk:k-svd-different-size}{
    In \hyperref[algo:K-basis]{$K$-SVD} one may generalize and assign a different number of basis elements to each of the $K$ clusters. The proof of \hyperref[th:convergence]{Theorem~\ref{th:convergence}} immediately extends to such a more general case, as in~\eqref{eq:k-svd-problem}, as the proof is independent on the number $m_j$ of columns each matrix $\Phi_j$ has, as long as those numbers
    $m_j$ are fixed.    
\qed}

\example{ex:suboptimality-SVD}{We consider again the parametric constrained least-squares problem defined in \hyperref[ex:PCA-basis]{Example~\ref{ex:PCA-basis}}. Using the same $M=1000$ samples, we run \hyperref[algo:K-basis]{$K$-SVD} for different values of $K$ (number of clusters) and $m$ (number of basis vectors), where the case $K=1$ corresponds to the single SVD decomposition~\eqref{eq:svd-W}.
    The clusters generated by \hyperref[algo:K-basis]{$K$-SVD} are separated by training $K$ neural one-to-all classifiers. Each classifier is a neural network composed by $3$ layers of $10$ neurons each with sigmoidal activation function $\frac{1}{1+e^{-x}}$, cascaded by a sigmoidal output function, corresponding to $281$ coefficients. For each cluster $h$, $h=1,\ldots,K$, if the training sample $\theta_i$ belongs to that cluster then the corresponding label $j_i=1$, otherwise $j_i=0$. The standard cross-entropy loss $-(j\log(\hat j)+(1-j)\log(1-\hat j))$ is used for training the network using the batch nonlinear programming solver implemented in the ODYS Deep Learning Toolset~\cite{ODYS-DLT}. The total CPU time to run \hyperref[algo:K-basis]{$K$-SVD} and train the neural classifiers in MATLAB R2020b ranges between $11$ and $45$ seconds on an Intel Core i7-8550U 1.99 GHz machine, with roughly $5$ to $10\%$ of the time spent to run \hyperref[algo:K-basis]{$K$-SVD}.
    
    The obtained results in terms of relative optimality error, relative error,  
    and maximum relative constraint violation are shown in \hyperref[fig:PCA-solution-error]{Figure~\ref{fig:PCA-solution-error}}.
    As expected, the quality of the approximation $z_r$ increases if more basis vectors $m$ and partitions $K$ are available.
\qed}

\begin{figure}[t]
    \centerline{\includegraphics[width=\hsize]{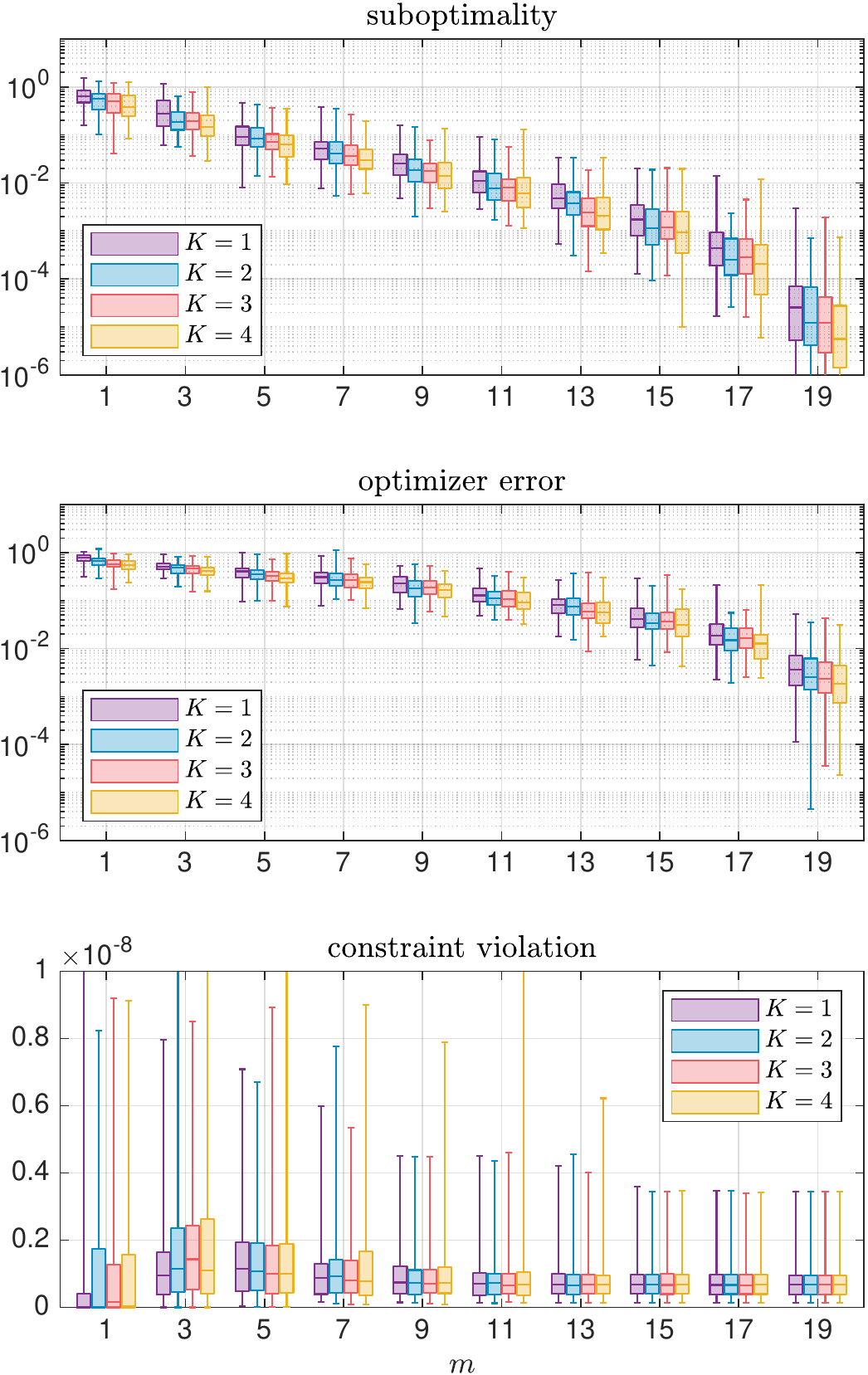}}
    \caption{Distribution of the errors induced by using Algorithm~\ref{algo:K-means}
        and one-to-all neural classifiers:
        relative optimality error (top plot), relative difference of optimizers (mid plot), relative
        maximum violation of the inequality constraints (bottom plot) 
        over 100 validation tests. The case $K=1$ corresponds to using a single SVD decomposition and corresponds to the results shown in Figure~\ref{fig:PCA-solution-error}.}
    \label{fig:k-PCA-solution-error}
\end{figure}

\subsection{Preservation of selected optimization variables}
\label{sec:z1}
The approximation method described in the previous sections aims at approximating
the reduced vector $s$ of variables obtained after eliminating the equality constraints~\eqref{eq:pCLS-eq}. On the other hand, an approximation error $\Delta s$ with respect to $s^*$  propagates on the original equality-constrained vector $z$ as $Q_2\Delta s$, see~\eqref{eq:z_of_s-a}. It might be desirable to reduce the approximation error on certain components of $z$ of main interest, say without loss of generality the first $k\leq m$ components $z_1=[z^1\ \ldots\ z^k]'$ of $z$. 

Using weighted low-rank approximation methods instead of SVD in \hyperref[algo:K-means:SVD]{Step~\ref{algo:K-means:SVD}} of Algorithm~\ref*{algo:K-means}
or in \hyperref[algo:K-basis:SVD]{Step~\ref{algo:K-basis:SVD}} of $K$-SVD to
take into account matrix $Q_2$ in~\eqref{eq:z_of_s} would not be convenient, as it would involve iterative procedures~\cite{SJ03}. Instead, we propose to compute a different QR factorization than in~\eqref{eq:QR-fact-C}. 

Let $z=\smallmat{z_1\\z_2}$,
$z_2=[z^1\ \ldots\ z^k]'$, $z_2\in\rr^{\ell-k}$. Accordingly, the equality constraints $Cz=e$ become $C_1z_1+C_2z_2=e$. Let us compute the QR factorization of $C_2'$
\begin{equation}
    C_2'=\bar Q\matrice{c}{\bar R_1\\0},\quad \bar Q=[\bar Q_1\ \bar Q_2]    
    \label{eq:QRfact-2}
\end{equation}
and consider the change of variables
\[
y=\matrice{c}{y_1\\y_2}=\matrice{cc}{I & 0\\0 & \bar Q'}z,\quad  z=\matrice{cc}{I & 0\\0 & \bar Q}y
\]
Then, by further splitting $y_2=[s_3'\ s_2']'$, $s_3\in\rr^{n_e}$, $s_2\in\rr^{\ell-k-n_e}$, we get $e=Cz=C_1y_1+[\bar R_1'\  0]\bar Q'[0\ \bar Q]\smallmat{y_1\\y_2}=C_1z_1+\bar R_1's_3$
from which we obtain
\begin{subequations}
    \beqar
    s_3&=&(\bar R_1')^{-1}(e-C_1z_1)\\
    s&=&\matrice{c}{z_1\\s_2}~\mbox{free}
    \label{eq:s=z1s2}
    \eeqar%
\end{subequations}
Finally, we get the following transformation from $s$ to $z$
\begin{equation}
    z=\matrice{c}{y_1\\\bar Qy_2}\matrice{c}{z_1\\\bar Q_1s_3+\bar Q_2s_2}=\bar Q_2s+\bar z
    \label{eq:z1-kept}
\end{equation}
where 
\[
Q_2\eqdef \matrice{cc}{I & 0\\-\bar Q_1(\bar R_1')^{-1}C_1 & \bar Q_2},
\quad \bar z=\matrice{c}{0\\\bar Q_1(\bar R_1')^{-1}e}
\]
Therefore, we can apply \hyperref[algo:K-basis]{$K$-SVD} or \hyperref[algo:K-means]{Algorithm~\ref{algo:K-means}} to the samples $s_i^*$ obtained from~\eqref{eq:s=z1s2} after substituting $z$ as in~\eqref{eq:z1-kept} in the constrained parametric LS problem. In order to emphasize that the components $z_1$ of $s$ are better approximated than the remaining components $s_2$, we can replace the samples $s_i^*$ with $\smallmat{\tau z_1^*\\s_2^*}$ in computing the bias terms and SVDs, where $\tau>1$ is a scalar parameter.
After executing \hyperref[algo:K-basis]{$K$-SVD} (or \hyperref[algo:K-means]{Algorithm~\ref{algo:K-means}}), the bias term $\phi_0$ (or $\phi_0^j$) needs to be scaled back by dividing its first
$k$ components by $\tau$.

\section{Variable-reduction methods for MPC}
We now want to adopt the methods developed in the previous sections to address the
pCLS problems that arise in model predictive control formulations, and to further refine
such methods to exploit the particular structure of those problems.

Let us assume that the dynamics of the process are modeled by the linear state-space model 
\begin{equation}
    x_{k+1}=\Aa_k(\theta)x_k+\BB_k(\theta)u_k
    \label{eq:dynamics}
\end{equation}
with $\Aa\in\rr^{n_x\times n_x}$ and $\BB\in\rr^{n_x\times n_u}$. Most often, model~\eqref{eq:dynamics} is obtained by linearizing a nonlinear model of the process around a nominal trajectory, and $x_k$, $u_k$ represent the predicted deviations from such a trajectory.
Typically the nominal state trajectory is the one obtained by applying the sequence of manipulated variables optimized at the previous MPC execution, starting from the current state. 

Let $z=[u_0'\ x_1\ \ldots\ u_{T-1}\ x_T]'$ be the vector of optimization variables collecting the sequence of manipulated variables $u_k\in\rr^{n_u}$ and the corresponding state variables $x_k\in\rr^{n_x}$ over a prediction horizon of future $T$ steps, with $z\in\rr^\ell$, $\ell=T(n_x+n_u)$.
Vector $\theta\in\rr^p$ collects the current estimate $x(t)$ of the plant to control, the current (and possibly future) reference signals to track, and other parameters affecting the prediction model, the performance index, and the constraints. 

The equality constraints~\eqref{eq:pCLS-eq} embedding model~\eqref{eq:dynamics} have the following band-matrix form
\begin{equation}
    \begin{array}{c} 
        \matrice{ccccccc}{\BB_0& -I    & 0   &    & \ldots  &         &0\\
            0 & \Aa_1   & \BB_1 & -I & 0       & \ldots  & 0\\
            \vdots & &     & \vdots  &  &  & \vdots\\
            0 & \ldots&     & 0  & \Aa_{T-1} & \BB_{T-1} & -I}
        z=\matrice{c}{\Aa_0x_0\\0\\\vdots\\0}\\
        \vspace*{-1em}
    \end{array}
    \label{eq:model-constraints}
\end{equation}
where in~\eqref{eq:model-constraints} we have omitted the possible
dependence of $\Aa_k$, $\BB_k$ on $\theta$ for simplicity. 
Constraints~\eqref{eq:model-constraints} are a special case of~\eqref{eq:pCLS-ineq} 
in which $n_e=Tn_x$, $\ell=T(n_x+n_u)$. We assume that $\rank(C)=Tn_x$.

In MPC problems, the cost function~\eqref{eq:pCLS-cost} is usually defined as
\begin{equation}
    \ba{rcl}
    A(\theta)&=&\blkdiag(R_{u_0},R_{x_1},\ldots,R_{u_{(T-1)}},R_{x_T})\\[1em]
    b(\theta)&=&\matrice{c}{R_{u0}u_{r0}\\R_{x1}x_{r1}\\\vdots\\R_{u(T-1)}u_{r(T-1)}\\
        R_{xT}x_{rT}}
    \ea
    \label{eq:linear-cost}
\end{equation}
where $\blkdiag()$ is the block-diagonal matrix of its arguments,
$A(\theta)\in\rr^{\ell\times\ell}$, $b(\theta)\in\rr^{\ell}$,  
$\|R_{u_k}(\bar u_k-u_k)\|_2^2$ is the penalty on inputs and $\bar u_{k}\in\rr^{n_u}$
is the input reference, and similarly $\|R_{x_k}(\bar x_k-x_{k})\|_2^2$ penalizes the deviation of the states from their reference $\bar x_{k}\in\rr^{n_x}$.
Penalties on outputs $\|R_{y_k}(\bar y_k-y_k)\|_2^2$, with $y_k=\CC_k x_k$, $y_k\in\rr^{n_y}$,
are a special case in which $R_{x_k}=R_{y_k}\CC_k$, and $R_{x_k}\bar x_{k}$ is replaced
by $R_{y_k}\bar y_{k}$ in~\eqref{eq:linear-cost}. Mixed input/state costs involving both $u_k$, $x_k$ could be also considered, which would lead to having off-diagonal blocks in $A(\theta)$.
Note that the weights in~\eqref{eq:linear-cost} may be rectangular matrices, for example in case some input is not weighted $R_{u_k}$ would have less than $n_u$ rows.

The inequality constraints~\eqref{eq:pCLS-ineq} can
collect $n_i$ constraints on inputs, input increments, 
states, and any other linear constraint involving a linear combination of such variables. In most practical MPC applications, it is customary to treat some of the inequality constraints as \emph{soft}, by introducing a vector $\zeta\in\rr^{n_\zeta}$ of additional slack variables as in~\eqref{eq:pCLS-ineq-soft}.

In deploying MPC laws in embedded platforms, solving problem~\eqref{eq:pCLS} efficiently in real time poses two main challenges:
\begin{itemize}
    \item [C1.] How to deal with equality constraints~\eqref{eq:pCLS-eq} and possibly
    eliminate them in a numerically robust way, reducing the number of optimization variables (excluding slacks) 
    from $\ell$ to $n=Tn_u$;
    \item [C2.] How to further reduce the number of degrees of freedom
    from $n$ to $m$, $m<n$, possibly sacrificing optimality in favor of lighter computations.
\end{itemize} 
Regarding (C1), the pCLS problem~\eqref{eq:pCLS} is typically reformulated in
the so-called \emph{condensed form} by replacing the predicted states $x_k$ 
with the corresponding prediction
\begin{equation}
    x_k=\prod_{i=0}^{k-1}\Aa_ix_0+\sum_{i=0}^{k-1}\prod_{j=i+1}^{k-1}\Aa_j\BB_iu_i
\label{eq:linear-response}
\end{equation}
where $\prod_{i=j}^{j+k}\Aa_i=\Aa_{j+k}\ldots \Aa_j$
if $k\geq 1$, or $\Aa_j$ if $k=0$, or the identity $I_{n_x}$ if $k<0$.
Such a condensing procedure allows one to recast~\eqref{eq:pCLS} to a smaller pCLS problem with only $n$ (or $n+n_\zeta$) variables. The main drawback of the condensing~\eqref{eq:linear-response} is its potentially poor numerical robustness: for example in case of unstable dynamics the substitution~\eqref{eq:linear-response} easily blows up numerically. 

A possible remedy is to prestabilize the dynamical system~\eqref{eq:dynamics} by setting
\begin{equation}
    u_k=\KK_k(\theta)x_k+u^c_k    
    \label{eq:prestabilizer}
\end{equation}
and then treat $u^c_0$, $\ldots$, $u^c_{T-1}$ as new optimization variables. While this can be easily accomplished
offline for linear time-invariant (LTI) systems, in case of linear parameter-varying (LPV)
or linear time-varying (LTV) systems the stabilizing set of feedback gains $\{\KK_k\}_{k=0}^{T-1}$ must be computed online for the given value of $\theta$. A way
to compute such gains is to solve a finite-horizon unconstrained linear-quadratic (LQ) optimal control problem
using the Riccati iterations
\begin{equation}
\ba{rcl}
    \PP_{k}&=&\QQ_k-\Aa_k'\PP_{k+1}\BB_k(\RR_k+\BB_k'\PP_{k+1}\BB_k)^{-1}\BB_k'\PP_{k+1}\Aa_k\\
    &&+\Aa_k'P_{k+1}\Aa_k\\
    \KK_{k}&=&-(\RR_k+\BB_k'\PP_{k}\BB_k)^{-1}\BB_k'\PP_{k}\Aa_k\\
\ea
\label{eq:DP}
\end{equation}
for $k=T-1,\ldots,0$, where $\QQ_k=R_{x_k}'R_{x_k}$, $\RR_k=R_{u_k}'R_{u_k}$,
and $\PP_T=R_{x_T}'R_{x_T}$.

Solution methods based on the \emph{non-condensed form}, in which the states $x_1$, $\ldots$, $x_T$ are kept as optimization variables, have the drawback of requiring general purpose sparse linear-algebra libraries to solve~\eqref{eq:pCLS} efficiently, or ad-hoc algorithms
exploiting the special structure~\eqref{eq:model-constraints}, for example when using interior-point methods~\cite{Wri19}. 
A drawback of such an approach comes when one wants to address the second challenge (C2), that is to reduce the number of optimization variables to simplify the mathematical programming problem to solve online, which alters the structure of~\eqref{eq:model-constraints}. 

The variable elimination method in \hyperref[sec:var_elim]{Section~\ref{sec:var_elim}}, based on the QR factorization of $C'(\theta)$, offers a numerically robust way to cope with (C1), even when the dynamics~\eqref{eq:dynamics} is unstable, while preserving the possibility to further reduce the degrees of freedom by means of advanced algorithms such as \hyperref[algo:K-basis]{$K$-SVD}. 

The next example shows the numerical properties of the different variable elimination methods mentioned above.

\begin{figure}[tb]
	\includegraphics[width=\hsize]{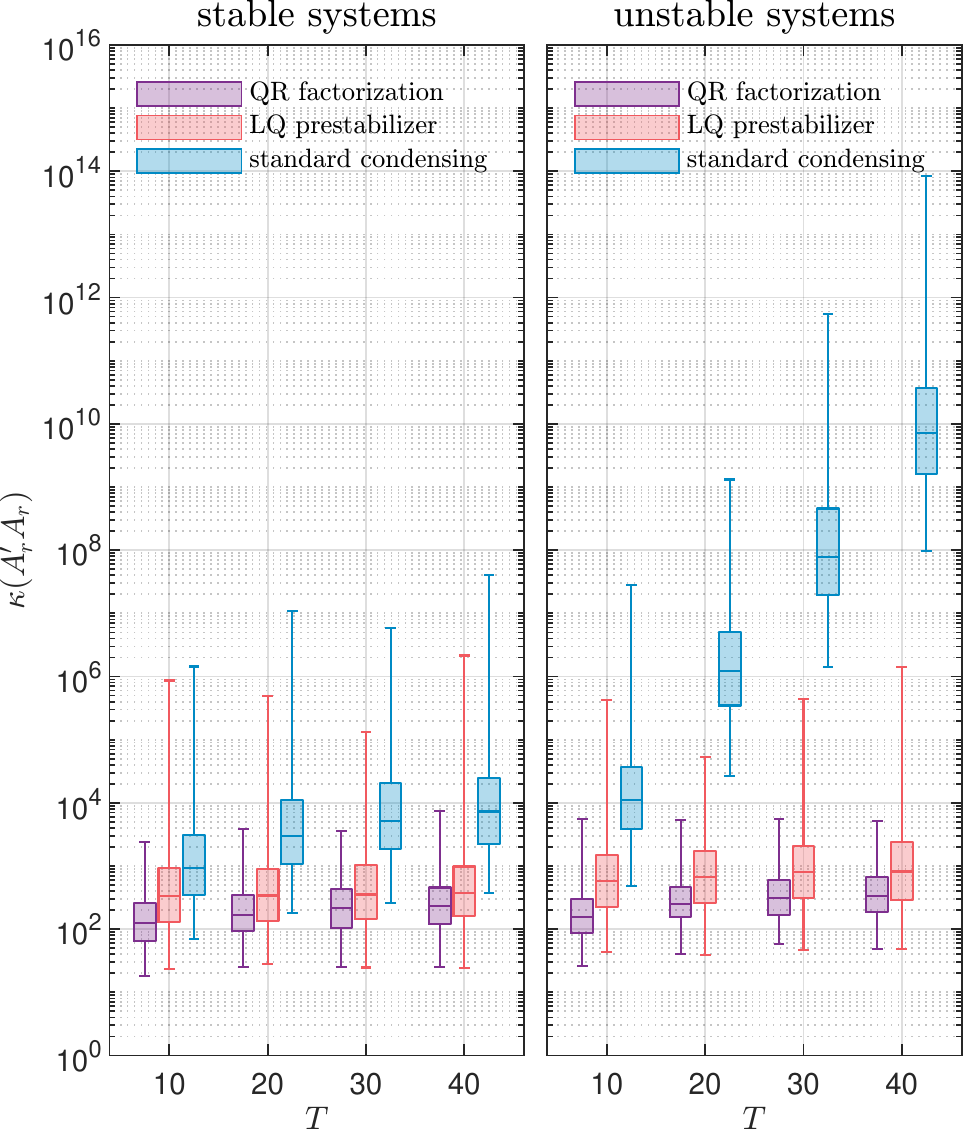}
	\caption{Distributions of the condition number of $A_r'A_r$, the Hessian matrix of the reduced LS problem after eliminating equality constraints~\eqref{eq:pCLS-eq}, comparing different condensing techniques. The distribution for each prediction horizon is obtained on a set of $1000$ random linear systems. Both stable (\emph{left}) and unstable (\emph{right}) systems are tested, the latter with real eigenvalues between $1$ and $1.25$. }
	\label{fig:qr_vs_standard_condensing}
\end{figure}

\example{ex:QR-LTI-unstable}{
Consider random linear systems with $n_x=5$ states, $n_u=3$ inputs,
and prediction horizons $T\in\{10,20,30,40\}$. We first consider \emph{stable} systems, with eigenvalues $\lambda(\Aa)$ such that $\lambda_i\sim\mathcal{U}(0.499,0.999),\, i=1,\dots,n_x$, and then \emph{unstable} systems with real eigenvalues $\lambda_i\thicksim\mathcal{U}(1,1.25),\,i=1,\dots,n_x$. For each prediction horizon $T$, a total of $1000$ stable and unstable random systems are generated. The weight matrices $R_{u_i},\,R_{x_i},\,i=1,\dots,T$ for input and state penalties are constant along the horizon, and change for each problem with their structure chosen such that $\rank(A)\sim\mathcal{D}(\floor{ {Tn_x\over 3}}+Tn_u,\ell)$, and their non-zero elements drawn from $\mathcal{U}(1,10)$. \hyperref[fig:qr_vs_standard_condensing]{Figure~\ref{fig:qr_vs_standard_condensing}}
shows the condition number of the Hessian matrix $A_r'A_r$ associated with the LS problem
\begin{equation}
	\min_s\,\,\frac{1}{2}\|A_rs-b_r\|_2^2\quad 	\st\,\,G_rs\leq g_r
	\label{eq:pCLS-generic}%
\end{equation}
obtained from~\eqref{eq:pCLS} by eliminating $Tn_x$ variables by standard condensing~\eqref{eq:linear-response} ($s=[u_0'\ \ldots\ u_{T-1}']'$), by prestabilization using
dynamic programming~\eqref{eq:DP} ($s=[(u^c_0)'\ \ldots\ (u^c_{T-1})']'$, see~\eqref{eq:prestabilizer}), and by QR factorization~\eqref{eq:QR-fact-C} ($s$ such that $z=\bar z+Q_2s$). The relation $\kappa(A_r'A_r)=\kappa^2(A)$ clearly holds. While
standard condensing numerically explodes (even for moderately unstable systems), both the LQ prestabilizer and the QR factorization deal effectively with unstable systems, although the latter is considerably more robust.\qed
}

Note that advanced pre-conditioning methods for $C(\theta)$, among which we cite geometric mean scale~\cite{F82}, can be used to improve the numerical robustness of the methods in \hyperref[ex:QR-LTI-unstable]{Example~\ref{ex:QR-LTI-unstable}}. However, such an improvement is only marginal and, more importantly, standard condensing would blow up numerically anyway.

\subsection{Efficient elimination of equality constraints}
\label{eq:QRimpl}
When applied to MPC settings, the numerical robustness of the elimination method in \hyperref[sec:var_elim]{Section~\ref{sec:var_elim}} comes at the expense of a substantially less efficient MPC routine if compared to~\eqref{eq:linear-response}. Indeed, a standard QR decomposition algorithm, with full $Q$ computation, introduces an $\mathcal{O}(5T^3n_x^3)$ complexity term. In this section we present an algorithm tailored to computing the QR decomposition of $C'(\theta)$ very efficiently. Afterwards, we derive a rigorous analysis of the flops (floating-point operations) required by the standard and the proposed QR condensing routines.

We start by noting that matrix $C(\theta)$ is very sparse. In such a scenario it is preferred to annihilate the columns of $R_1$ by means of Givens rotations, which selectively introduce zeros below the diagonal, one
element at a time \cite{Wilki88,GVL13}. Let $x=\begin{bmatrix}a & b\end{bmatrix}'$ be a vector with $\|x\|_2=r>0$; a Givens rotation is a $2$-by-$2$ matrix $G_r$ such that $G_r\smallmat{a\\b}=\smallmat{r\\0}$. It is easy to show that the $Q$ and $R_1$ factors of the decomposition of $C'(\theta)$ are sparse as well and such that
\begin{subequations}
	\begin{align}
		R_1&=\begin{bmatrix}
			U_1 & \tilde R_1 & 0&  0\\
			0 & \ddots & \ddots & 0\\
			\vdots & \ddots & \ddots & \tilde R_{T-1}\\
			0 & \dots & 0 &   U_T\\
		\end{bmatrix}\label{eq:r_sparsity}\\
		\begin{bmatrix}
			Q_1\, |\, Q_2
		\end{bmatrix}&=\begin{bmatrix}
			{Q}_{1,1} & \dots & Q_{1,T}\, |\, {Q}_{2,1} & \dots & {Q}_{2,T}
		\end{bmatrix}
		\label{eq:q_sparsity}
	\end{align}
	\label{eq:qr_sparsity}%
\end{subequations}
where $U_i\in\rr^{n_x\times n_x}$ is an upper triangular matrix, $\tilde R_j\in\rr^{n_x\times n_x}$ is dense, $Q_{1,i}=\begin{bmatrix}
	\tilde{Q}_{1,i}' & \tilde U_i' & 0
\end{bmatrix}'$, $Q_{2,i}=\begin{bmatrix}
	0 & L_i' & \tilde{Q}_{2,i}'
\end{bmatrix}'$, with $\tilde{Q}_{1,i}\in\rr^{(n_ui+n_x(i-1))\times n_x}$ dense, $\tilde{U}_i\in\rr^{n_x\times n_x}$ upper triangular, $L_i\in\rr^{n_u\times n_u}$ lower triangular, $\tilde{Q}_{2,i}\in\rr^{(n_u+(n_x+n_u)(T-i))\times n_u}$ dense, and $i\!=\!1,\dots,T$, $j\!=\!1,\textrm{\dots}, T-1$. The sparsity pattern of~\eqref{eq:model-constraints} and~\eqref{eq:qr_sparsity} provides the basis for developing Algorithm \ref*{algo:QRmod}, that we call \hyperref[algo:QRmod]{QR-MPC}. Its computational analysis is described by the following lemma, which highlights the efficiency of the method.

\lemma{lem:qrcomplexity}{Let $C\in\rr^{Tn_x \times \ell}$ be the matrix of equality constraints for problem~\eqref{eq:pCLS} defined as in~\eqref{eq:model-constraints}. \hyperref[algo:QRmod]{QR-MPC} computes the QR decomposition $\smallmat{Q_1 &Q_2}\smallmat{R_1 \\ 0}=C'$ used for the numerically robust condensing of problem \eqref{eq:pCLS}, and enjoys the following properties:
	\begin{itemize}
		\item [$i$)] The serial complexity $\chi_{\text{QR}}$ of \hyperref[algo:QRmod]{QR-MPC}, namely the total number of flops required for termination without any parallel computation, is
		\begin{equation}
			\chi_{\text{QR}}\!=\!T^3(n_x^2n_u\!+\!n_xn_u^2)\!+\!3T^2n_un_x(2n_x\!+\!n_u\!+\!1)\!+\!\mathcal{O}(T)
			\label{eq:chi_QR}
		\end{equation}
		\item [$ii$)] \hyperref[algo:QRmod]{QR-MPC} saves a total of
		\begin{equation}
			\chi_{\text{QRS}}-\chi_{\text{QR}}=T^3(5n_x^3+11n_x^2n_u+5n_xn_u^2)-\mathcal{O}(T^2)
			\label{eq:chi_QRS-chi_QR}
		\end{equation}
		flops, with $\chi_{\text{QRS}}$ the complexity of a standard non-economy QR decomposition algorithm relying on Givens rotations for columns annihilation.
	\end{itemize}
}
\begin{proof}
	$i)$ The complexity of the factorization can be expressed as the sum of terms $\chi_{\text{QR}}=\chi(R)+\chi(Q)$, where  we denote by $\chi(Y)$ the complexity of computing matrix $Y$. Clearly $\chi(R)$ comes from the iterative execution of \hyperref[step:QR-Rstep]{Step~\ref{step:QR-Rstep}}, and $\chi(Q)$ from \hyperref[step:QR-Qstep]{Step~\ref{step:QR-Qstep}}. Let $mp(2n-1)$ be the complexity of the matrix vector product $M=Y\bar D$ with $Y\in\rr^{m\times n}$, $\bar D\in\rr^{n\times p}$, then
	\begin{equation}
		\chi(R)=\sum_{j=1}^{Tn_x}\sum_{i=j+1}^{r_1}6\Big(n_x\min\big(\floor{\tfrac{j-1}{n_x}}+2,T\big)-j+1\Big)
	\label{eq:chir_pre}
	\end{equation}
	Let us define the integral operators $\II_0(n)\triangleq\int_{0}^{n}\tau\,d\tau$, $\II_1(n)\triangleq\int_{0}^{n}(\tau+1)\,d\tau$ and $\II_2(n)\triangleq\int_{0}^{n}(\tau+1)^2\,d\tau$; we can get rid of the \emph{minimum} and \emph{integer part} operations in \eqref{eq:chir_pre} by rearranging the summations such that
	\begin{align}
		&\chi(R)\!=\!6\sum_{k=1}^{T}\sum_{j=1}^{n_x}\sum_{i=1}^{kn_u}(2n_x\!-\!j\!+\!1)\!-\!6\sum_{j=1}^{n_x}\sum_{i=1}^{Tn_u}(n_x\!-\!j\!+\!1)\approx\notag\\
		&\approx 6n_u\Big(\II_1(T)\,(2n_x^2\!-\!\II_0(n_x))\!-\!T(n_x^2\!-\!\II_0(n_x))\Big)\!=\notag\\
		&=\frac{9}{2}T^2n_x^2n_u
		\label{eq:Rflops}
	\end{align}
	On the other hand, the complexity of $\chi(Q)$ is 
	\begin{equation}
		\chi(Q)\!=\!\sum_{j=1}^{Tn_x}\sum_{i=j+1}^{r_1}\!6\big(j+n_u(1+\floor{\tfrac{j-1}{n_x}})-(n_x+n_u)\floor{\tfrac{i-j-1}{n_u}}    \big)
		\label{eq:Qflops_init}
	\end{equation}
	Let us define $w=n_x+n_u$, by expanding the summations in \eqref{eq:Qflops_init} similarly to the contribution of $\chi(R)$, $\chi(Q)$ can be rewritten as
	\begin{align}
		&\chi(Q)\!=\!6\sum_{k=1}^{T}\sum_{j=1}^{n_x}\sum_{h=1}^{k}\sum_{i=1}^{n_u}\big( (k-h)w+n_u+j\big)\approx\notag\\ &\approx6n_u\sum_{k=1}^{T}\Big(k^2n_xw-kn_x^2+k\II_1(n_x)-w\sum_{j=1}^{n_x}\II_0(k)\Big)\approx\notag\\
		& \approx 6n_u\big({n_x\over 2}w\II_2(T)+n_x(1-{n_x\over 2})\II_1(T)\big)=\notag\\
		&=\!T^3(n_x^2n_u\!+\!n_xn_u^2)\!+\!3Tn_xn_u\big(T(\frac{n_x}{2}\!+\!n_u\!+\!1)+(n_u\!+\!3)\big)
		\label{eq:Qflops}
	\end{align}
	From~\eqref{eq:Rflops} and~\eqref{eq:Qflops} we get that $\chi_{\text{QR}}\!=\!\chi(R)\!+\!\chi(Q)\!=\!T^3(n_x^2n_u\!+\!n_xn_u^2)\!+\!3T^2n_un_x(2n_x\!+\!n_u\!+\!1)\!+\!\mathcal{O}(T)$, which proves $i)$.
	
	$ii)$ Let $C\in\rr^{m\times n}$ be a full-rank matrix, the serial complexity $\chi_{\text{QRS}}$ of a standard QR decomposition algorithm based on Givens rotations, and forming $Q$ matrix excplicitly (non-economy version) is
	\begin{align}
		\chi_{\text{QRS}}=\sum_{j=1}^{n}\sum_{i=j+1}^{m}6(m\!+n\!-\!j\!+\!1)\approx6n(m^2\!-\!m\!-\!\frac{n^2}{6})
		\label{eq:stdQR}
	\end{align}
	By replacing $m=T(n_x+n_u)$ and $n=Tn_x$ in \eqref{eq:stdQR} we obtain
	\begin{align}
		\chi_{\text{QRS}}= T^3(5n_x^3\!+\!6n_xn_u^2\!+\!12n_x^2n_u)\!-\!6T^2(n_x^2\!+\!n_un_x)
		\label{eq:chi_QRS}
	\end{align}
	From \eqref{eq:chi_QR} and \eqref{eq:chi_QRS} it follows that $\chi_{\text{QRS}}-\chi_{\text{QR}}=T^3(5n_x^3+11n_x^2n_u+5n_xn_u^2)-\mathcal{O}(T^2)$, which proves $ii)$.
\end{proof}

\begin{theorem}
	\label{thr:cond_complexity}
Let~\eqref{eq:pCLS-cost}--\eqref{eq:pCLS-eq} be the equality constrained least-squares formulation of an MPC problem with $n_x$ states, $n_u$ inputs, a prediction horizon of $T$ steps, $C(\theta)$, $e(\theta)$ as in~\eqref{eq:model-constraints}, and $A(\theta)$, $b(\theta)$ as in~\eqref{eq:linear-cost}. Let 
	\begin{equation}
		\min_s\,\,\frac{1}{2}\| A_r(\theta)s-b_r(\theta)\|^2_2
	\end{equation}
be the reformulation of~\eqref{eq:pCLS-cost}--\eqref{eq:pCLS-eq} in \emph{condensed form} 
after removing equalities~\eqref{eq:pCLS-eq}, where $s\in\rr^{Tn_u}$, $A_r(\theta)\in\rr^{\ell\times Tn_u}$ and $b_r(\theta)\in\rr^\ell$. Then, the serial complexity  for computing $A_r(\theta)$, $b_r(\theta)$ is
	\begin{itemize}
		\item [$i$)]  $\chi_s$ in case of \emph{standard condensing}, with
		\begin{equation}
			\chi_s=T^2(n_x^2n_u-{n_xn_u\over 2})+\mathcal{O}(T)
			\label{eq:chi_sf}
		\end{equation}
		\item [$ii$)] $\chi_q$ in the case of \emph{QR condensing}, with
		\begin{align}
			\chi_q\!= &T^3(n_x^2n_u\!+\!n_xn_u^2)\!+\!T^2(n_x^2(6n_u\!+\!2)\!+\!n_u^2(3n_x\!+\!1)+\notag\\&+\!2n_un_x\!-\!2(n_x\!+\!n_u))\!+\!\mathcal{O}(T)
			\label{eq:chi_qf}
		\end{align}
	\end{itemize}
\end{theorem}
\begin{proof}
	$i)$ We first note that the states replacement~\eqref{eq:linear-response} is equivalent to impose the coordinate transformation $z=Fs+f$ with $F\in\rr^{\ell\times Tn_u}$, $f\in\rr^{\ell}$ defined as 
	\begin{align}
		F\!=\!\begin{bmatrix}
			I_{n_u} & 0 & \dots & \dots & 0\\
			\BB &0 & \dots & \dots &  0\\
			0 & I_{n_u} & 0 & \dots &  0\\
			\Aa\BB & \BB & 0 & \dots & 0 \\
			\vdots & \ddots & \ddots & \ddots & \vdots\\
			0 & \dots  & \dots & 0 & I_{n_u} \\
			\Aa^{T-1}\BB & \Aa^{T-2}\BB & \dots & \Aa\BB & \BB\\[1pt]
		\end{bmatrix}\!,\,
		f\!=\!\begin{bmatrix}
			0\\
			\Aa\\
			0\\
			\Aa^2\\
			\vdots\\
			0\\
			\Aa^T   
		\end{bmatrix}x_0
		\label{eq:stdcond_Zz}
	\end{align}
	where in~\eqref{eq:stdcond_Zz} we have omitted the dependence of $\Aa$  and $\BB$ from $k$, and therefore the condensed problem corresponds to~\eqref{eq:pCLS-r} after having replaced $Q_2\gets F$ and $\bar{z}\gets f$. Clearly $CF=0$ holds, and $f$ solves~\eqref{eq:pCLS-eq}. The complexity $\chi^s_1$ associated with forming $f$ is equivalent to $T$ matrix-vector products of dimension $n_x$, thus
	\begin{equation}
		\chi^s_1=Tn_x(2n_x-1)
	\end{equation}
	Let $M=Y\bar D$ be a matrix-vector product with $Y\in\rr^{n_x\times n_x}$, $\bar D\in\rr^{n_x\times n_u}$; then, constructing $F$ has a complexity $\chi_2^s$ equivalent to $T-1$ repetitions of such product, that is
	\begin{equation}
		\chi^s_2=(T-1)n_xn_u(2n_x-1)
	\end{equation}
	The cost $\chi^s_3$ for computing $A(\theta)F$ comes instead from $\sum_{i=1}^{T}i$ repetitions of $M=Y\bar D$, such that
	\begin{equation}
		\chi^3_s=T^2(n_x^2n_u-{n_xn_u\over 2}+Tn_xn_u(2n_x-1)
	\end{equation}
	and the complexity of $b-F\bar{z}$ is
	\begin{align}
		\chi^s_4=Tn_x(2n_x-1)+\ell
	\end{align}
	We finally prove $i)$ by simply computing $\chi_s=\sum_{i=1}^{4}\chi_i^s$.
	
	$ii)$ From \hyperref[lem:qrcomplexity]{Lemma \ref{lem:qrcomplexity}} we know the complexity of factorizing $C(\theta)$, and thus we are left with computing the flops required to form the cost function \eqref{eq:pCLS-cost-r}. We can exploit the sparsity of $R_1$ in \eqref{eq:r_sparsity} when solving the linear system $\bar{s}=(R'_1)^{-1}e$ by forward substitution. That is equivalent to run the update step $\bar{s}_i\gets(e_i-R'_{i,k:i-1}\bar{s}_{k:i-1})/R_{i,i}$, with $k=1+\max(0,n_x\floor{{i-1 \over n_x}}\!-\!1)$ and $i=1,\dots,n$, which takes 
	\begin{equation}
		\chi_1^q=n_x^2(3T-2)+Tn_x
	\end{equation}
	flops to be executed. 
	For the complexity $\chi^q_2$ of matrix-vector product $\bar z\!=\!Q_1\bar{s}$ we exploit instead the sparsity of $Q_1$ in \eqref{eq:q_sparsity}. Let us recall that the flops required for computing $c\!=\!Yd$ with $Y$ sparse, is well approximated by $2n_z$, with $n_z$ the number of non-zero elements of $Y$. Then, we get the complexity
	\begin{equation}
		\chi^q_2=n_x^2(T^2+2T-1)+n_x(n_u(T^2-1)-2T)
	\end{equation}
	The  structure of $A(\theta)$ in \eqref{eq:linear-cost} is first used to derive  the complexity $\chi^q_3$ for computing $b(\theta)-A(\theta)\bar{z}$ that is
	\begin{equation}
		\chi^q_3=2(n_u^2+n_x^2)T+\ell
	\end{equation}
	and then, combined with $Q_2$ sparsity, to derive the complexity $\chi^q_4$ of the matrix-vector product $A(\theta)Q_2$, which is
	\begin{equation}
		\chi^q_4\!=\!T^2(n_x^2\!+\!n_u^2\!-\!2(n_x\!+\!n_u))\!+\!T(2(n_u^2\!+\!n_x^2)\!-\!n_u\!-\!n_x)
	\end{equation}
	Computing $\chi_q=\chi_{\text{QR}}+\sum_{i=1}^4\chi_i^q$ proves $ii)$.
	
\end{proof}

\remark{rmk:qr_improve}{The memory allocation of \hyperref[algo:QRmod]{QR-MPC} can be reduced by overwriting the upper triangular part of $C(\theta)'$ with $R_1$, thus saving the allocation space for $R$. This is easily achieved by computing the plane rotation on the couple $(C(\theta)_{k,i-1},C(\theta)_{k,i})$, and applying it to $C'(\theta)$ at Step \ref*{step:QR-Rstep}. Note that any robust computation of $c$ and $s$ factors, see \cite{BDKM02} for instance, can be used in place of Steps~\ref*{step:c-Qstep} and~\ref*{step:s-Qstep} to avoid over/underflow.	
	In addition, \hyperref[algo:QRmod]{QR-MPC} factorization can be modified to take explicitly into account possible zero-entries of $\Aa$ and $\BB$. If $|C_{k,i-1}|<\epsilon_0$, we get the rotation  $G_r=\smallmat{ 0 & -1\\
		1 & 0}$ which corresponds to a signed swap of $C(\theta)$ and $Q$ columns, saving therefore the flops needed for the matrix-vector products of Steps~\ref*{step:QR-Rstep} and~\ref*{step:QR-Qstep}.	
\qed}

\hyperref[thr:cond_complexity]{Theorem \ref{thr:cond_complexity}} proves the serial complexity of reformulating the cost function~\eqref{eq:pCLS-cost} in the presence of the equality
constraints~\eqref{eq:pCLS-eq}, when either \emph{standard} or \emph{QR} condensing is adopted. The flops for reformulating inequalities~\eqref{eq:pCLS-ineq} are instead neglected because most often in MPC applications, states and inputs are constrained by simple bounds, that is $G\!=\!\begin{bmatrix}I_{\ell}&-I_{\ell}\end{bmatrix}'$, in which case computing~\eqref{eq:pCLS-ineq-r} costs approximately $n_i$ flops, regardless of the condensing method used. If instead constraints involving the linear combination of states and/or inputs are present, for instance output constraints, the cost of reformulating all the closed half-spaces defining a specific constraint along the prediction horizon is $n_u(2n_p-1)\sum_{i=1}^{T}i$ flops, assuming the worst-case in which the constraint is imposed at each time step, where $n_p\in\{1,\dots,n_x+n_u\}$ is the number of variables involved in the definition of the constraint itself.

We finally note that the computational assessment of \hyperref[thr:cond_complexity]{Theorem \ref{thr:cond_complexity}} remains valid even in case a vector $\zeta\in\rr^{n_{\zeta}}$ of slack variables is introduced for softening (some of) the constraints. Indeed, when reformulating the pCLS problem~\eqref{eq:pCLS} such that $\begin{bmatrix}z'&\zeta'\end{bmatrix}'$ is the extended vector of optimization variables, equalities~\eqref{eq:pCLS-eq} are replaced by 
\begin{equation}
\begin{bmatrix} C & 0_{n_{\zeta}} \end{bmatrix}\begin{bmatrix}z\\\zeta\end{bmatrix}=e
\end{equation}
and hence the QR factorization of $C'$ in \eqref{eq:QR-fact-C} becomes
\begin{equation}
\begin{bmatrix}
	C'\\0_{n_{\zeta}}
\end{bmatrix}=\begin{bmatrix}
	Q & 0\\
	0 & I_{n_{\zeta}}
\end{bmatrix}\begin{bmatrix}
R\\0'_{n_{\zeta}}
\end{bmatrix}
\label{eq:extQR}
\end{equation}
From \eqref{eq:extQR} we know that the extended optimizer vector for the pCLS problem~\eqref{eq:pCLS-r} without equality constraints is $\begin{bmatrix}s'&\zeta'\end{bmatrix}'$, which means slack variables are kept as free variables, and one can derive the reduced pCLS by applying \hyperref[algo:QRmod]{QR-MPC} to \eqref{eq:pCLS} and extend the optimizer vector afterwards.

\begin{algorithm}[t]
	\caption{\textcolor{mycol1}{QR-MPC} factorization for efficient condensing.}
	\label{algo:QRmod}
	~~\textbf{Input}: Matrix $C(\theta)\in\rr^{Tn_x\times \ell}$ of equality constraints~\eqref{eq:pCLS-eq}, $n_u$, $n_x$, $T$, zero-detection tolerance $\epsilon_0$.
	\vspace*{.1cm}\hrule\vspace*{.1cm}
	\begin{enumerate}[label*=\arabic*., ref=\theenumi{}]
		\item $Q\gets I_{\ell}$;
		\item $R\gets C(\theta)'$;
		\item \textbf{for} $j=1,\ldots,Tn_x$ \textbf{do}:
		\begin{enumerate}[label=\theenumi{}.\arabic*., ref=\theenumi{}.\arabic*]
			\item $k\gets \floor{\frac{j-1}{n_x}}$; 
			\item $r_1\gets j+n_u(k+1)$; 
			\item $r_2\gets n_x\min(k+2,T)$; 
			\item \textbf{for} $i=r_1,\ldots,j+1$ \textbf{do}:
			\begin{enumerate}[label=\theenumii{}.\arabic*., ref=\theenumii{}.\arabic*]
				\item $q_1\gets1+(n_x+n_u)\floor{\frac{i-j-1}{n_u}}$;
				\item \textbf{if} $|R_{i,k}|>\epsilon_0$   \textbf{do}:
				\begin{enumerate}[label=\theenumii{}.\arabic*., ref=\theenumii{}.\arabic*]
					\item $c\gets {R_{i-1,k} / \sqrt{R_{i-1,k}^2+R_{i,k}^2}}$\label{step:c-Qstep};
					\item $s\gets -{R_{i,k} / \sqrt{R_{i-1,k}^2+R_{i,k}^2}}$\label{step:s-Qstep}; 
					\item $R_{i-1:i,j:r_2}\gets\begin{bmatrix}
						c & s\\ -s &c
					\end{bmatrix}'R_{i-1:i,j:r_2}$ 	\label{step:QR-Rstep};
					\item	$Q_{q_1:r_1,i-1:i}\gets Q_{q_1:r_1,i-1:i}\begin{bmatrix}
						c & s\\ -s &c
					\end{bmatrix}$\label{step:QR-Qstep};
				\end{enumerate}
			\end{enumerate}            
		\end{enumerate}
		\item \textbf{end}.
	\end{enumerate}
	\vspace*{.1cm}\hrule\vspace*{.1cm}
	~~\textbf{Output}: Orthogonal matrix $Q\in\rr^{\ell\times\ell}$ and upper triangular matrix  $R\in\rr^{\ell\times n_e}$ such that $QR=C(\theta)'.$
\end{algorithm}

\example{ex:QR-efficiency}{From~\hyperref[ex:QR-LTI-unstable]{Example~\ref{ex:QR-LTI-unstable}} it is clear that a numerically robust method to eliminate equalities is~\emph{mandatory} when the system dynamics are unstable. Indeed, condensing by means of~\eqref{eq:linear-response} blows-up numerically even in double precision. When instead  $\Aa_k$, $k=1,\dots,T$ are stable, standard condensing is appealing because of the reduced throughput required to compute the $(F,f)$ pair and form \eqref{eq:pCLS-generic}, see the results of \hyperref[thr:cond_complexity]{Theorem~\ref{thr:cond_complexity}}. Here we want to show that even in this scenario, a robust elimination of equalities may still be valuable despite the increase in computational complexity. Consider parameters $n_x$, $n_u$, $\rank(A)$, $\lambda_j$, $\Aa_{k}$, $\BB_{k}$, $R_{u_{k\!-\!1}}$, $R_{x_k}$, $k\!=\!1,\dots,T$, $j\!=\!1,\dots,n_x$ defined according to the random pCLS setup of \hyperref[ex:QR-LTI-unstable]{Example~\ref{ex:QR-LTI-unstable}}, and restricted to the sole case of stable dynamics. We enforce box constraints on all the inputs, that is $u^-\le u_k\le u^+$, $k\!=\!1,\dots,T$, with $u^-_{(j)}\sim\mathcal{U}(-\delta,\delta)$, $u^+_{(j)}\sim\mathcal{U}(u^-_{(j)},u^-_{(j)}\!+\!\delta)$, $j=1,\dots,n_u$ and $\delta\!=\!1$. We also constrain the states such that $x^-\le x_i\le x^+$, $i\!=\!1,\dots,T$ and 
\begin{align}
&x^-_{(j)}=\begin{cases}
	\sim\mathcal{U}(-\delta,\delta)\quad \text{if } x^-_{(j)} \text{ is imposed}\\
	-\infty\quad \text{otherwise}
\end{cases}\notag\\
&x^+_{(j)}=\begin{cases}
	\sim\mathcal{U}(x^-_{(j)},x^-_{(j)}\!+\!\delta)\quad \text{if } x^+_{(j)} \text{ is imposed}, x^-_{(j)}\neq-\infty\\		\sim\mathcal{U}(-\delta,\delta)\quad \text{if } x^+_{(j)} \text{ is imposed}, x^-_{(j)}\equiv-\infty\\
	\infty\quad \text{otherwise}
\end{cases}
\end{align}
with $j=1,\dots,n_x$. For each problem we randomly select which upper and lower bounds on the states are enforced in such a way that the total number of constraints is $n_i=2n_u+m_x$, with $m_x\sim\mathcal{D}(\floor{{n_x\over 2}},\floor{{4n_x\over 3}})$. We consider two procedures for solving such generated MPC problem. The first, denoted ``MPC\textsubscript{STD}'' is based on standard condensing and solves the reduced pCLS problem~\eqref{eq:pCLS-generic} constructed by applying the variable transformation $z=Fs+f$, see~\eqref{eq:stdcond_Zz}. In the second one, which is denoted by ``MPC\textsubscript{QR}'', one has to first factorize $C'$ (by means of the \hyperref[algo:QRmod]{QR-MPC} algorithm), compute the transformation $z=Q_2s+Q_1(R_1')^{-1}e$, and compute the matrices defining~\eqref{eq:pCLS-generic}, which is ultimately solved. In both cases we solve the reduced pCLS problem for $s$ by means of the Alternating Direction Method of Multipliers (ADMM); the reader is referred to \cite{BPC11} for mathematical details. 
Let us define  $s^{j+1}$ and $h^{j+1}$ the sequential primal updates on the directions $s$ and $h$ at the $j$-th iteration, and $w^{j+1}$ the dual update, the steps of an over-relaxed ADMM applied to~\eqref{eq:pCLS-generic} are:
\begin{align}
	\label{eq:ADMMiter}
	s^{j+1}&=(A_r'A_r+\rho G_r'G_r)^{-1}(A'_rb_r-\rho G_r'(s^j+w^j-g_r))\notag\\
	h^{j+1}&=\begin{bmatrix}\alpha_1h^j
	-\alpha(G_rs^{j+1}-g_r)-w^j	
	\end{bmatrix}_+\\
w^{j+1}&=w^j+\alpha(G_rs^{j+1}+h^{j+1}-g_r)+\alpha_1(h^{j+1}-h^{j})\notag
\end{align}
with $\alpha\in(0,2)$ the relaxation parameter, and $\rho>0$ the weight on the augmented Lagrangian.
\begin{figure}[tb]
	\includegraphics[width=\hsize]{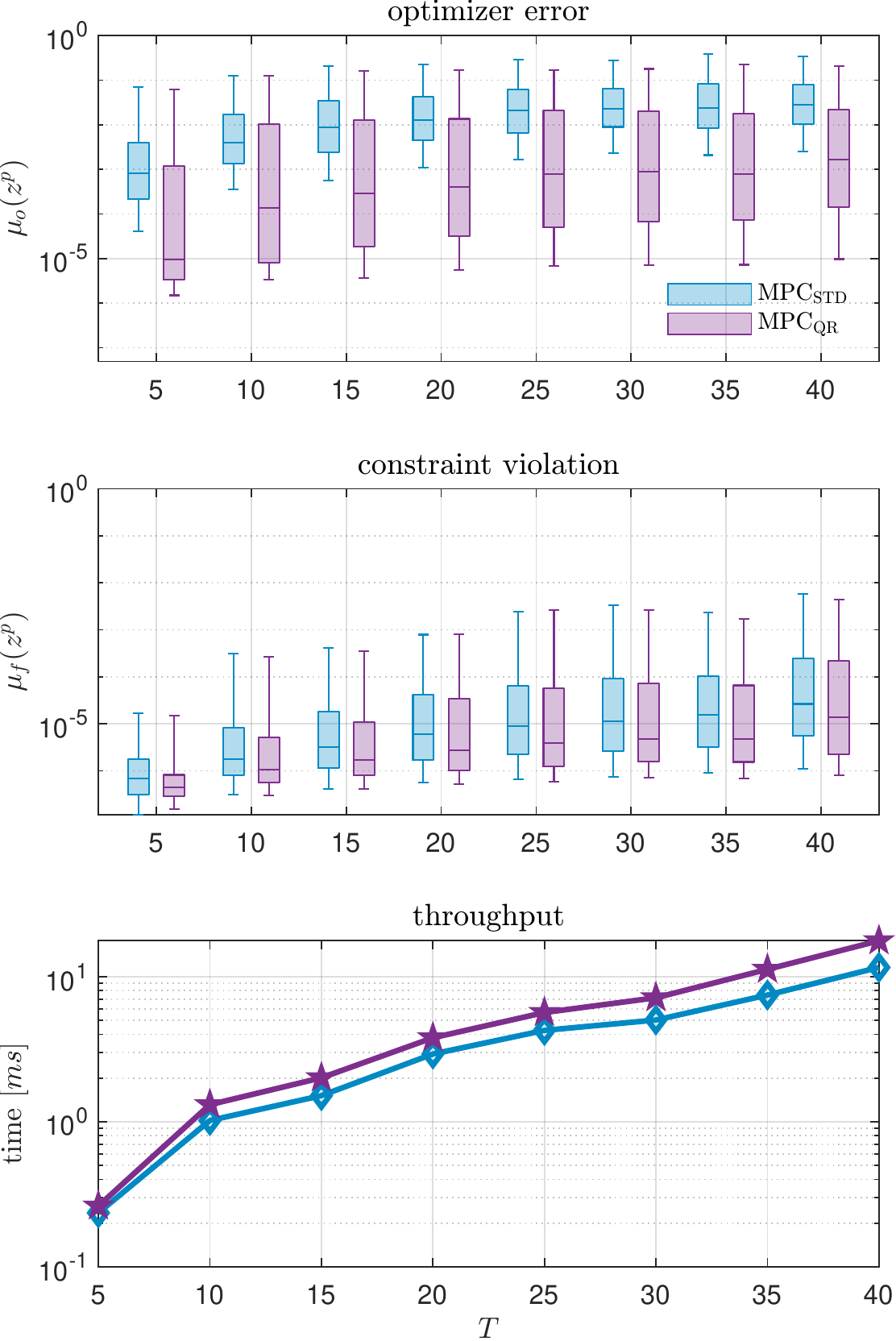}
	\caption{Comparison MPC\textsubscript{STD} and MPC\textsubscript{QR} in terms of solution quality and computational load on a set of $n_t=1000$ random \emph{stable} MPC problems. Algorithms are coded in single precision floating-point arithmetic, and the reduced pCLS \eqref{eq:pCLS-generic} is solved by a fixed number of ADMM iterations ($p=200$), see~\eqref{eq:ADMMiter}. Distribution of the optimizer error (\emph{top}), distribution of the maximum constraint violation (\emph{mid}) and shifted geometric mean ($h_t=10$) of the time required to both condense and solve the reduced problem (\emph{bottom}) are reported.}
	\label{fig:mpc_benchmark}
\end{figure}
We consider a \emph{single precision} floating-point implementation of MPC\textsubscript{STD} and MPC\textsubscript{QR}, with ADMM running for a fixed amount  $p=200$ of iterations. The quality of the solution reached after $p$ iterations is evaluated in terms of the optimizer error $\mu_o(z^p)$ and the violation $\mu_f(z^p)$ of the constraints with respect to the original problem~\eqref{eq:pCLS}, defined as:
\begin{equation}
	\mu_o(z^p)= \dfrac{\|z^*-z^p\|}{\|z^*\|},\quad \mu_f(z^p)=\max\left\{J\begin{bmatrix}
		|Cz^p-e|\\(Gz^p-g)_+
	\end{bmatrix}\right\}
\end{equation}
where $J=\diag(d_J)^{-1}$, $d_J=
    [\|[C_1\ e_1]\|_2$ $\dots$ $\|[C_{n_e}\ e_{n_e}]\|_2$
    $\|[G_1\ g_1]\|_2$ $\dots$ $\|[G_{n_i}\ g_{n_i}]\|_2]'$.
Clearly we have that $z^p=Fs^p+f$ for MPC\textsubscript{STD}, and $z^p=Q_2s^p+Q_1(R_1')^{-1}e$ for MPC\textsubscript{QR}. \hyperref[fig:mpc_benchmark]{Figure~\ref{fig:mpc_benchmark}} shows the distribution of optimality and feasibility when solving $n_t=1000$ random MPC problems for each $T$. The bottom plot shows also the total computational load of both MPC routines, obtained as the \emph{shifted geometric mean}
\begin{equation}
	\nu(t)=\left(\prod_{i=1}^{n_t} (t_i+h_t)\right)^{1/n_t}
	\label{eq:SGM}
\end{equation}
where $t\in\rr^{n_t}$ is the array collecting the execution time over the $n_t$ different CLS problems, and $h_t$ is a shift parameter. We assume $h_t=10$ in this example. MPC\textsubscript{QR} improves both the quality metrics, especially the optimizer error which is reduced by more than one order of magnitude. Such results undoubtedly show that the robust elimination is not only an option to deal with unstable systems, but it is beneficial even  when standard condensing is reliable. The price for such robustness is an increase in the total throughput of the MPC routine, as shown in the figure. We stress that in our setup the computational load of solving~\eqref{eq:pCLS-generic} is identical for MPC\textsubscript{STD} and MPC\textsubscript{QR}, being $p$ and the steps of a single ADMM iteration the same. This means that the time difference is only due to the increased complexity of computing the variable transformation and forming~\eqref{eq:pCLS-generic}. The relative impact on the total time is therefore determined by the overhead of the optimization algorithm used to solve~\eqref{eq:pCLS-generic}. Moreover, algorithms whose convergence rate depends on $\kappa(A_r)$ will converge faster for an MPC\textsubscript{QR} implementation, with a consequent shrinking of the computational gap.

\begin{figure}[tb]
	\includegraphics[width=\hsize]{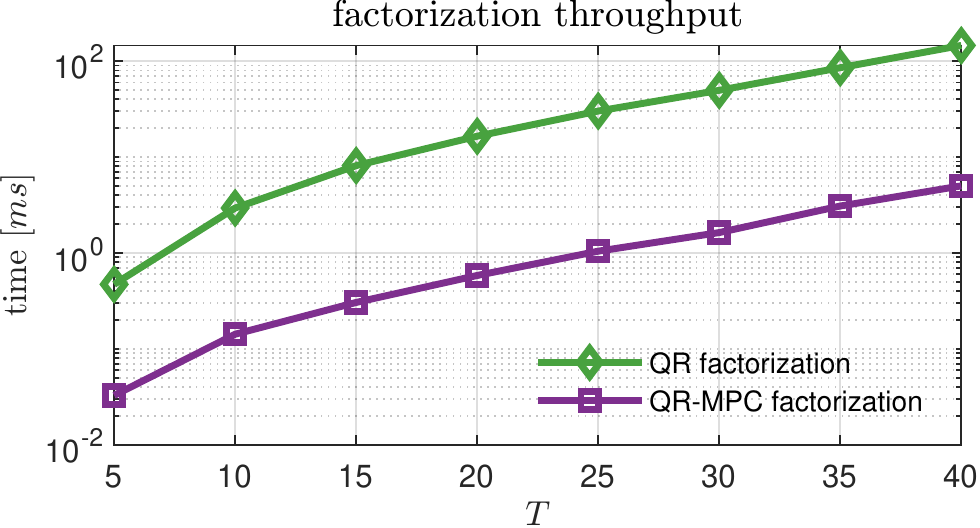}
	\caption{Computational evaluation of the proposed variable elimination method.
		\hyperref[algo:QRmod]{QR-MPC} (\emph{violet}) is compared to a standard (\emph{green}) non-economy QR factorization, see~\cite[Algorithm 5.2.4]{GVL13}. The time shown for each horizon $T$ is the shifted geometric mean~\eqref{eq:SGM}, with $h_t=10$, computed over $n_t=1000$ random CLS problems.
	}
	\label{fig:qr_benchmark}
\end{figure}
Lastly, we highlight the  efficiency of the~\hyperref[algo:QRmod]{QR-MPC} algorithm by comparing its throughput with respect to a standard non-economy QR decomposition based on Givens rotations, see \cite[Algorithm 5.2.4]{GVL13}. \hyperref[fig:qr_benchmark]{Figure~\ref{fig:qr_benchmark}} shows the timing results on the same set of MPC problems used in~\hyperref[fig:mpc_benchmark]{Figure~\ref{fig:mpc_benchmark}}. Note that for the implementation of~\hyperref[algo:QRmod]{QR-MPC} we have taken into account throughput and memory efficiency ideas proposed in \hyperref[rmk:qr_improve]{Remark~\ref{rmk:qr_improve}}. This is the building block of the MPC\textsubscript{QR} routine and it is indeed the main contribution to the computation time shown by the bottom plot of \hyperref[fig:mpc_benchmark]{Figure~\ref{fig:mpc_benchmark}}. By comparing the two figures we can conclude that without \hyperref[algo:QRmod]{QR-MPC} the computational gap between MPC\textsubscript{STD} and MPC\textsubscript{QR} would have been worse by up to an order of magnitude.
\qed}

\subsection{Reduction of the number of variables}
In MPC applications, it is often observed that reducing the number of degrees of freedom to $m<n$ does not compromise closed-loop performance, while it instead simplifies the on-line optimization problem. As mentioned in \hyperref[sec:intro]{Section~\ref{sec:intro}}, a way that is commonly used in MPC to reduce the number of optimization variables is \emph{move blocking}, which consists of keeping the input signal $u_k$ constant between prediction steps $k_i$ and $k_{i+1}-1$, $i=0,\ldots,m_u$, with $m_u$ the number of steps where the input signal is free to change, $k_0=0$ and  $k_{m_u}=T$. In this way, the number of optimization variables is reduced from $n$ to $m=m_un_u$. Frequently, $k_i=i$ for $i=0,\ldots,m_u-1$, i.e., the input signal is free to move within a \emph{control horizon} of $m_u$ steps and then is frozen afterwards until the end of the prediction horizon $T$. For example, having a control horizon $m_u\leq 5$ is enough in most MPC problems, due to the receding-horizon mechanism of MPC. Reference governors~\cite{BCM98,GKT95} are an extreme case in which, thanks to the presence of a prestabilizing feedback loop, using the control horizon $m_u=1$ to optimize the reference signals achieve the task of fulfilling constraints and maintaining good closed-loop performance.

\emph{Basis functions} have been suggested to reduce the number of degrees
of freedom by parametrizing the input sequence as
\begin{equation}
    u_k=\sum_{i=1}^{m} v_i\phi_i(k)
    \label{eq:blocking-moves-phi}
\end{equation}
where $\{\phi_i(\cdot)\}_{i=1}^m$ is a given basis. Blocking moves can be seen as a special case in which the basis functions are  the unit-step function $\one(k-k_i)=1$ for $k\geq k_i$, or zero otherwise,
\begin{equation}
    u_k=\sum_{i=1}^{m} v_i\one(k-k_i)
    \label{eq:blocking-moves}
\end{equation}
The use of Laguerre polynomials to parameterize the input sequence was investigated in~\cite{KR13}, which also examines other choices of orthogonal functions.
Laguerre polynomials were also used to parameterize
both the input and state sequences~\cite{MD19}, in a way that these are invariant to time shifts
and hence amenable for warm-starting the new optimization problem with the shifted solution of the previous problem.

Note that the use of basis functions as in~\eqref{eq:s-basis} becomes necessary when the free optimization variables $s$ have no system theoretical meaning, such as in case the QR method described in \hyperref[eq:QRimpl]{Section~\ref{eq:QRimpl}} is used to eliminate equality constraints, for which there is no direct and intuitive method like blocking moves to reduce the number of free variables, unless blocking-moves are introduced before removing equality constraints.

As we work in discrete-time over a finite horizon of $T$ steps, ``basis functions'' are nothing else than just vectors $\phi_i\in\rr^{n}$. Therefore, finding a basis 
of $m$ elements to parameterize the input sequence $\{u_k\}_{k=0}^{T-1}$ is equivalent to finding a matrix $\Phi\in\rr^{Tn_u\times m}$ with full column-rank. The SVD-based method described in \hyperref[sec:basis-SVD]{Section~\ref{sec:basis-SVD}}, or alternatively its generalization in \hyperref[sec:K-SVD]{Section~\ref{sec:K-SVD}},
can be immediately applied for this task. The approach described in \hyperref[sec:z1]{Section~\ref{sec:z1}} can
be also applied to preserve the components of the optimization vector corresponding to the applied command input $u_0$ as much as possible with respect to the original solution of~\eqref{eq:pCLS}.
Variable reduction through PCA of the condensed Hessian was also proposed to simultaneously reduce the optimizer size and limit the ill-conditioning effects of unstable systems. 
However, this needs an online SVD  when $A$ or $C$ depend on $\theta$, similarly to \cite{OW14,RGSF04,UMJ12}, and the factorization is performed on an ill-conditioned matrix. 

Regarding the feasibility of the reduced pCLS problem, we can apply the techniques presented in \hyperref[sec:feasibility]{Section~\ref{sec:feasibility}}. In particular, if a feasible input and state trajectory $z_f$ is available, then the corresponding vector $s_f$ can computed and one can set $\phi_0=s_f$ or use the parameterization~\eqref{eq:s-basis-feas}.

\section{Examples of application to MPC problems}

We apply the methods developed in the previous sections to the classical benchmark problem of MPC of a continuous stirring tank reactor (CSTR)~\cite{BMR20}. The system has $n_x=2$ states (the temperature of the reactor $T_r$ [K] and the concentration $C_A$ [kgmol/m$^3$] of reactant $A$) and $n_u=1$ input (coolant temperature $T_c$ [K]), see~\cite{SMED10}. The goal is to make $C_A$ track a given reference $C_A^{\text{ref}}$ by applying a linear parameter-varying MPC
with sample time $T_s=0.5$~h. Accordingly, the nonlinear dynamics are linearized around the current state $x(k)$ and previous input $u(k-1)$, then converted to discrete-time by first-order Euler approximation. 

We assume that the feed-stream concentration $C_{A_f}$ and temperature $T_f$, that are measured disturbances, are kept constant, namely $C_{A_f}$ = 10 kgmol/m$^3$
and $T_f$ = 298.15 K. Units will be omitted since now on where obvious.
The following constraints are imposed on the manipulated input $T_c$:
\[
    \ba{c}
    285.15\leq T_{c,k}\leq 312.15\\
    -3\leq T_{c,k}-T_{c,k-1}\leq 3\\
    \ea
\]
In order to handle constraints and weights on input increments $\Delta u_k=u_k-u_{k-1}$, the $2^{\rm nd}$-order model is extended with the additional dynamics $u_k=u_{k-1}+\Delta u_k$. We choose a prediction horizon $T=20$, preview on the reference signal of
one step (namely, $C_{A,k}^{\text{ref}}$ and  $C_{A,k+1}^{\text{ref}}$ are known at time step $k$), weight matrix $R_{\delta u}=0.01$ on input increments, $R_x=\smallmat{0 &0&0\\0&0 &1\\0&0&0}$ on the extended state $[x_k'\ u_{k-1}]'$. 

As shown in \hyperref[tab:J]{Table~\ref{tab:J}}, without degradation of closed-loop performance we also limit the control horizon to $n=3$ free moves, namely $\Delta u_k=0$ for all $k=n,n+1,\ldots,T-1$, where closed-loop performance is quantified by the following index
\begin{equation}
    J=\sum_{k=0}^{N} (C_A(k)-C_{A,k}^{ref})^2+0.01(T_{c,k}-T_{c,k-1})^2
    \label{eq:cost_J}
\end{equation}
and $N$ is the total number of closed-loop simulation steps.

Starting from the steady-state initial condition 
$u_{-1}=298.15$ and $x_0=[311.267\ 8.5695]'$, we apply the MPC controller
defined above and obtain the closed-loop results depicted in \hyperref[fig:CSTR-trajectories]{Figure~\ref{fig:CSTR-trajectories}}, corresponding to the cost $J_{\rm exact,3}$ reported in \hyperref[tab:J]{Table~\ref{tab:J}}.

\begin{table}
        \begin{center}
        \begin{tabular}{l|l|l}
            performance index & value & MPC setting \\[.5em]\hline
            $J_{\rm exact,2}$ &  567.8602 &  $n=2$, exact solution\\
            $J_{\rm exact,3}$ & 320.1679 & $n=3$, exact solution\\
            $J_{\rm exact,4}$ & 319.9293 & $n=4$, exact solution\\
            $J_{\rm exact,10}$ & 319.6873& $n=10$, exact solution\\
            $J_{\rm exact,20}$ & 319.6846& $n=20=T$, exact solution\\\hline
            $J_{\rm svd}$ &  328.3266 &  $n=3$, $m=2$ (single SVD)\\
            $J_{\rm ksvd}$ & 320.1677 &  $n=3$, $m=2$ (K-SVD)\\\hline
        \end{tabular}
        \end{center}
        \caption{Closed-loop performance index $J$ computed as in~\eqref{eq:cost_J} for different MPC solution methods.         \label{tab:J}
        }
\end{table}

\begin{figure}[t]
    \includegraphics[width=\hsize]{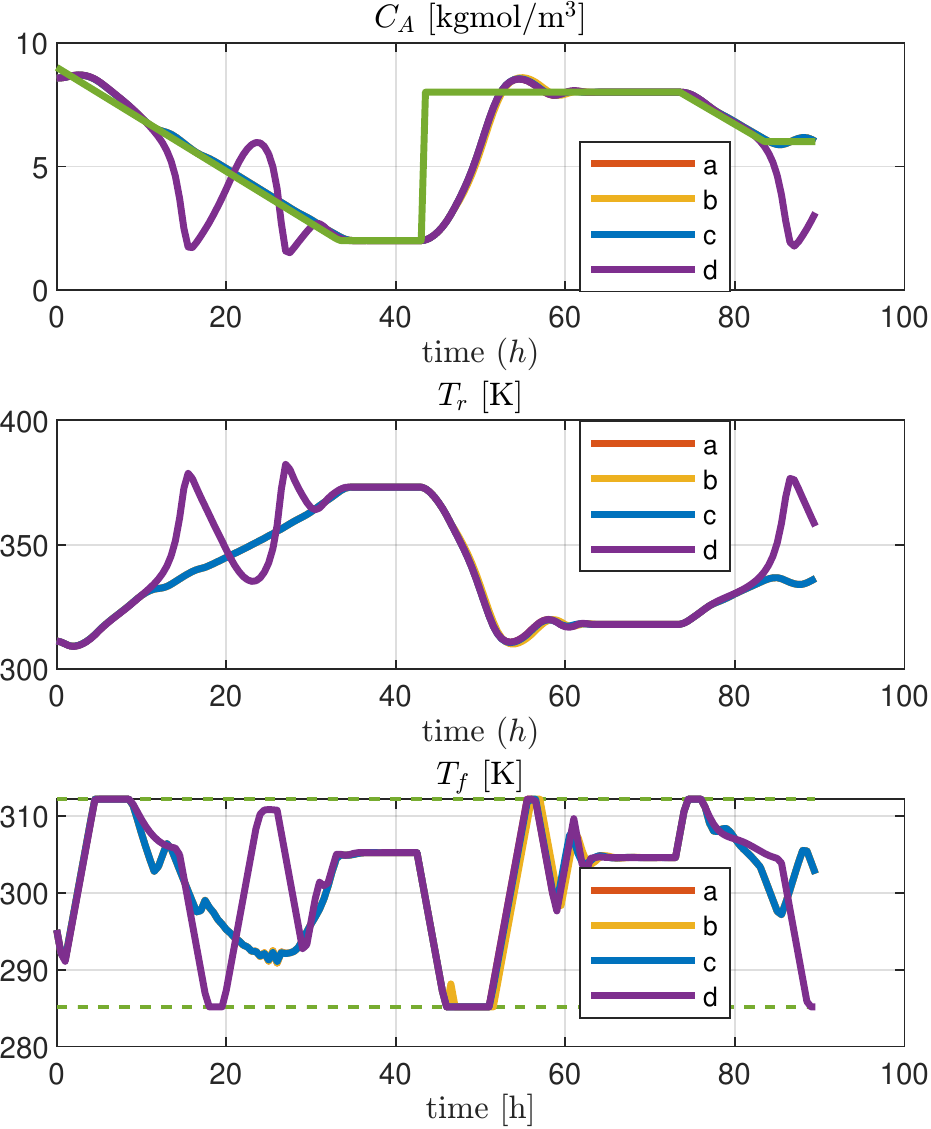}
    \caption{Closed-loop MPC results: (a) exact solution with control horizon $n=3$;
 (b) single SVD and (c) \hyperref[algo:K-basis]{$K$-SVD} ($K=10$) with two basis vectors, solving approximately the MPC problem with control horizon $n=3$; (d) exact solution with control horizon $n=2$.
The exact solution with  $n=3$ and the $K$-SVD solution are almost indistinguishable.}
    \label{fig:CSTR-trajectories}
\end{figure}

In order to reduce the number of optimization variables from $n=3$ to
$m=2$, we generate a dataset of $M=10,\!000$ samples of vector $\theta=[T_{r,k}\ C_{A,k}\ T_{c,k-1}\ C_{A,k}^{\text{ref}}\ C_{A,k+1}^{\text{ref}}]'$. The samples are generated by running a closed-loop simulation under the MPC controller designed
above with random step signal sampled from the uniform distribution between $2$ and $9$ kgmol/m$^3$ as reference signal $C_A^{\rm ref}$, where at each sample step the set-point has probability 10\% of switching. For each sample $\theta_i$, the QR factorization~\eqref{eq:QRfact-2} is applied to the MPC problem formulation as a constrained least-squares problem, so that the sequence $u^*_i=[\Delta u_k^*\ \Delta u_{k+1}^*\ \Delta u_{k+2}^*]_i'$ of optimal input increments is computed as in~\eqref{eq:z_of_s-a} from the solution $z_i^*=[z_{1i}^*\ s_{2i}^*]'$ of the reduced-order problem~\eqref{eq:pCLS-r}.
This is augmented with a slack variable $\zeta$ as in~\eqref{eq:pCLS-ineq-r-v-soft}
to avoid any infeasibility, which is heavily penalized by adding $(10^5\zeta)^2$
in the least-squares problem. As suggested in \hyperref[sec:z1]{Section~\ref{sec:z1}}, SVDs are applied to samples $\smallmat{\tau z_{1i}^*\\s_{2i}^*}$, with $\tau=20$ in order to favor the reconstruction of the first input increment $z_1^*$. 

We first compute a single SVD decomposition to find a basis $\Phi$ and offset $\phi_0$. The closed-loop simulation results obtained by minimizing with respect
to $v\in\rr^2$ are also reported in \hyperref[fig:CSTR-trajectories]{Figure~\ref{fig:CSTR-trajectories}}, corresponding to the performance index $J_{\rm svd}$ reported in \hyperref[tab:J]{Table~\ref{tab:J}}.

Next, we apply the $K$-SVD approach, running  \hyperref[algo:K-basis]{Algorithm~\ref{algo:K-basis}} with $K=10$.
The sizes $M_j$ of the obtained clusters are 12, 550, 2758, 155, 7, 506, 1111, 1036, 3610, 255. For each $j=1,\ldots,K$, we train a one-to-all neural classifier
with three hidden layers of 10, 10, and 6 neurons, respectively, with sigmoidal activation function, cascaded by a sigmoidal output function, corresponding to 281 coefficients to learn per classifier. The cross-entropy loss is used during training, which is accomplished in 1271 seconds on the same machine used in \hyperref[ex:suboptimality-SVD]{Example~\ref{ex:suboptimality-SVD}} using the ODYS Deep Learning Toolset~\cite{ODYS-DLT}.

The corresponding closed-loop simulation results are depicted in \hyperref[fig:CSTR-trajectories]{Figure~\ref{fig:CSTR-trajectories}}, corresponding to the performance index $J_{\rm ksvd}$
reported in \hyperref[tab:J]{Table~\ref{tab:J}}, which is indistinguishable from $J_{\rm exact}$.
The index of the basis function chosen at each step by selecting the neural classifier with the largest value is depicted in \hyperref[fig:MPC-ksvd-index]{Figure~\ref{fig:MPC-ksvd-index}}, while \hyperref[fig:MPC-ksvd-partition]{Figure~\ref{fig:MPC-ksvd-partition}} shows the section in the $(T_r,C_A)$ space of the five-dimensional partition induced by the classifiers,
for the remaining coordinates set to $u_{k-1}=299.9943$, $C_{A,k}^{\rm ref}=5.5121$, $C_{A,k+1}^{\rm ref}=5.5121$ (these values are the average values computed on the training dataset). The figure also shows the values encountered
during the closed-loop simulation reported in \hyperref[fig:CSTR-trajectories]{Figure~\ref{fig:CSTR-trajectories}}.

\begin{figure}[t]
    \includegraphics[width=\hsize]{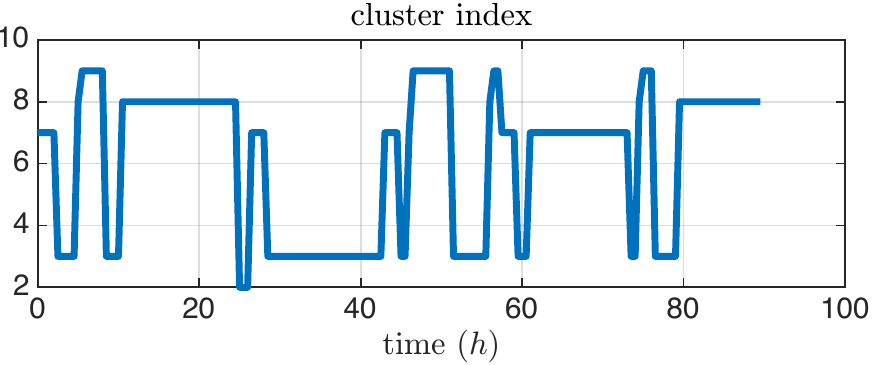}
    \caption{Index of basis chosen when applying MPC based on \hyperref[algo:K-basis]{$K$-SVD}.}
    \label{fig:MPC-ksvd-index}
\end{figure}
\begin{figure}[t]
    \includegraphics[width=\hsize]{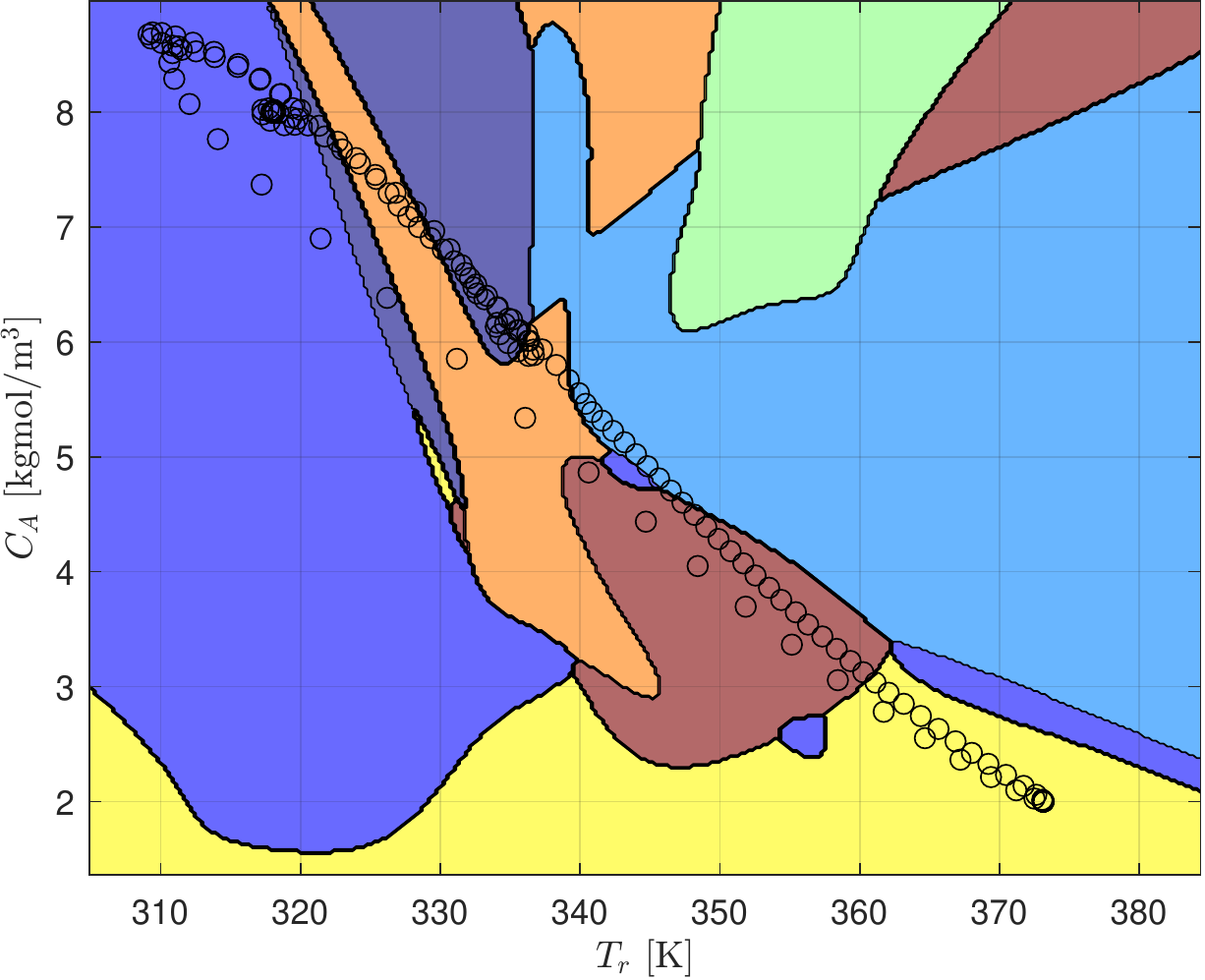}
    \caption{Section in the $(T_r,C_A)$ space of the five-dimensional partition induced by the neural classifiers for separating the clusters identified by K-SVD. The
    values encountered during the closed-loop simulation reported in Figure~\ref{fig:CSTR-trajectories} are shown as black circles.}
    \label{fig:MPC-ksvd-partition}
\end{figure}

Finally, we test how the MPC algorithm performs without SVD approximations for the same number of degrees of freedom, that is by setting
the control horizon $n=m=2$. The results are also shown in Figure~\ref*{fig:CSTR-trajectories},
corresponding to the closed-loop index $J_{\rm exact,2}$ reported
in Table~\ref*{tab:J}. 

\section{Conclusions}
In this paper we have investigated numerical methods for reducing the number of variables
in pCLS problems, such as those that are encountered in MPC formulations. The $K$-SVD algorithm is a general method that extends linear PCA to handle parameter-dependent sample vectors, still keeping the resulting approximation of the vectors a linear combination of the principal components, so that the reduced problem remains a pCLS. We have also shown that the QR factorization is a much more numerically-stable approach than standard condensing, and provided a specialized QR-based equality elimination scheme that exploits the special structure of pCLS' arising from MPC. 

Further research will be devoted to address issues of recursive feasibility and closed-loop stability of MPC schemes in which the degrees of freedom are reduced by using basis functions. 
Moreover, although we have addressed least-squares problems, our approach can be extended to other parameter-dependent optimization problems, such as to QP's, which would be an immediate extension, but also to other convex and nonconvex problem classes.

\bibliographystyle{ieeetran}
\bibliography{input_parameterization_bib}

\begin{IEEEbiography}[{\includegraphics[width=1in,height=1.25in,clip,keepaspectratio]{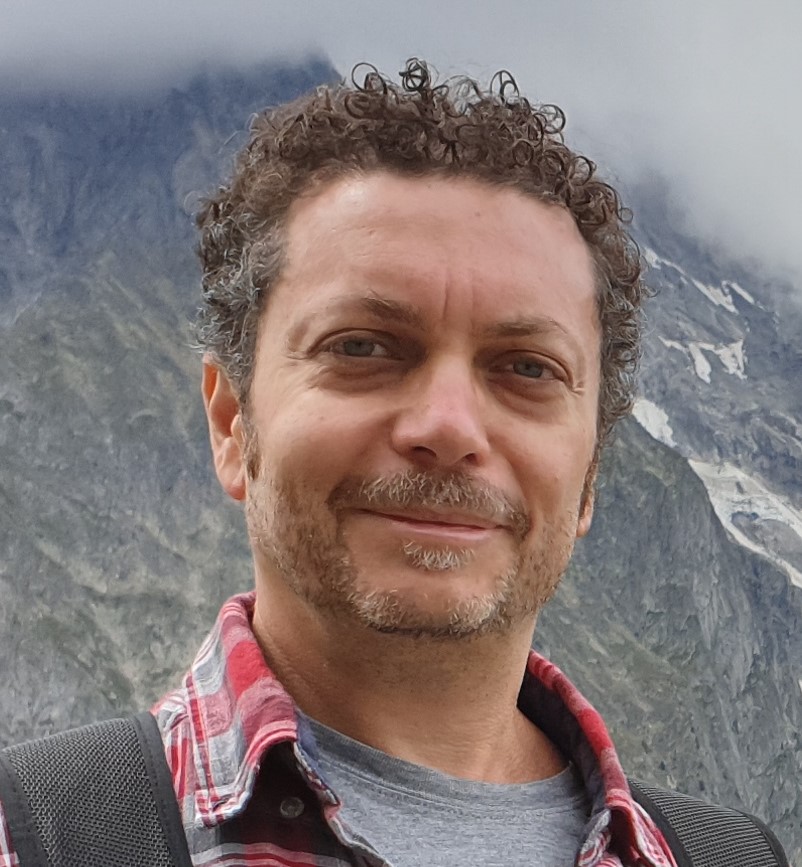}}]%
    {Alberto Bemporad} received his Master's degree in Electrical Engineering in 1993 and his Ph.D. in Control Engineering in 1997 from the University of Florence, Italy. In 1996/97 he was with the Center for Robotics and Automation, Department of Systems Science \& Mathematics, Washington University, St. Louis. In 1997-1999 he held a postdoctoral position at the Automatic Control Laboratory, ETH Zurich, Switzerland, where he collaborated as a senior researcher until 2002. In 1999-2009 he was with the Department of Information Engineering of the University of Siena, Italy, becoming an Associate Professor in 2005. In 2010-2011 he was with the Department of Mechanical and Structural Engineering of the University of Trento, Italy. Since 2011 he is Full Professor at the IMT School for Advanced Studies Lucca, Italy, where he served as the Director of the institute in 2012-2015. He spent visiting periods at Stanford University, University of Michigan, and Zhejiang University. In 2011 he cofounded ODYS S.r.l., a company specialized in developing model predictive control systems for industrial production. He has published more than 350 papers in the areas of model predictive control, hybrid systems, optimization, automotive control, and is the co-inventor of 16 patents. He is author or coauthor of various software packages for model predictive control design and implementation, including the Model Predictive Control Toolbox (The Mathworks, Inc.) and the Hybrid Toolbox for MATLAB. 
    He was an Associate Editor of the IEEE Transactions on Automatic Control during 2001-2004 and Chair of the Technical Committee on Hybrid Systems of the IEEE Control Systems Society in 2002-2010. He received the IFAC High-Impact Paper Award for the 2011-14 triennial and the IEEE CSS Transition to Practice Award in 2019. He is an IEEE Fellow since 2010.
\end{IEEEbiography}

\begin{IEEEbiography}[{\includegraphics[width=1in,height=1.25in,clip,keepaspectratio]{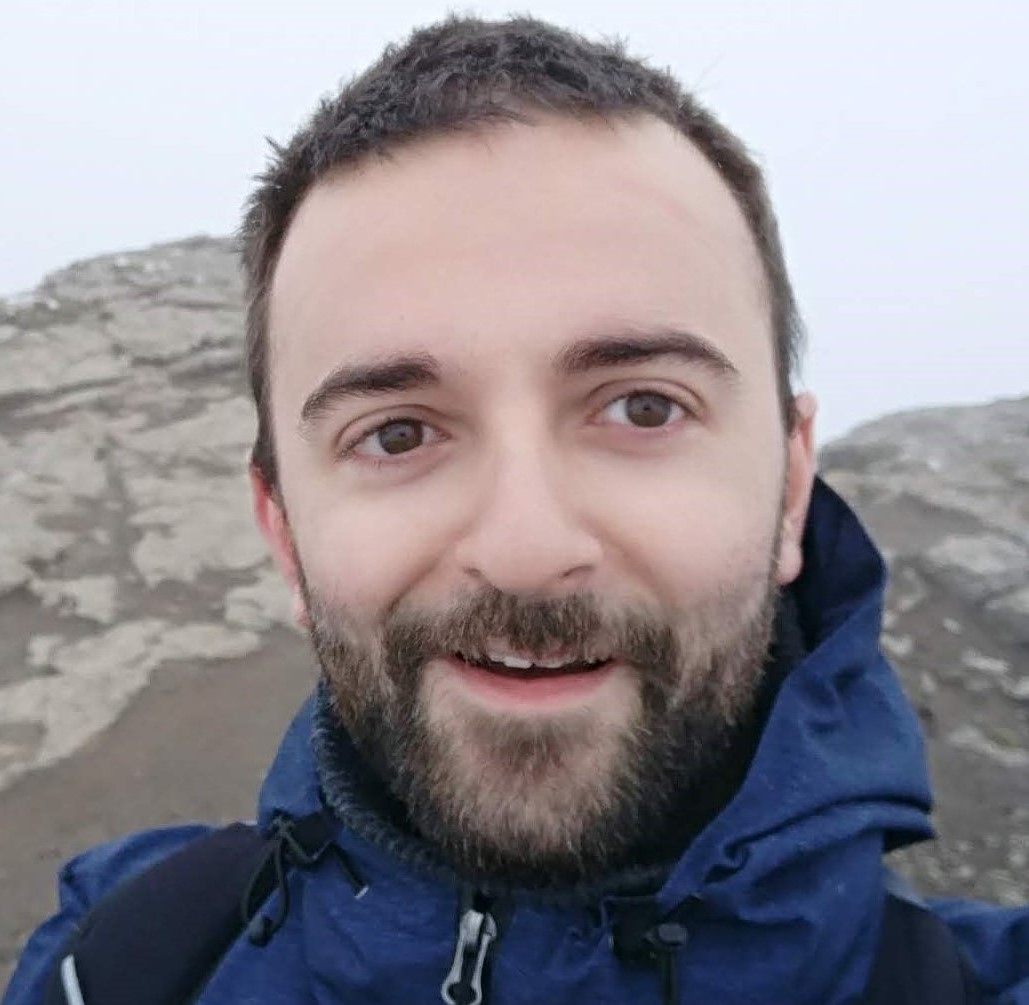}}]%
        {Gionata Cimini} received his Master's degree in Computer and Automation Engineering in 2012 and his Ph.D. degree in Information Engineering, in 2017, from Universit\`a Politecnica delle Marche, Italy. He has been a guest Ph.D. scholar at the IMT School for Advanced Studies Lucca, Italy, and a visiting Ph.D. scholar at University of Michigan, USA. He held research assistant positions at the Information Department, Universit\`a Politecnica delle Marche and at the Automotive Research Center, University of Michigan. In 2016, he was a contract employee at General Motors Company, USA. In 2017 he joined ODYS S.r.l.,  where he is currently the technical manager for numerical optimization. His research interests include model predictive control, embedded optimization, and their application to problems in the automotive, aerospace, and power electronics domains.
\end{IEEEbiography}

\end{document}